\documentclass[12pt,a4paper]{article}
\usepackage{amsmath}
\usepackage{amsfonts}
\usepackage{amsthm}
\usepackage{amssymb}

\usepackage{enumerate}
\theoremstyle{plain}
\newtheorem{theo}{Theorem}[section]
\newtheorem{lem}[theo]{Lemma}
\newtheorem{prop}[theo]{Proposition}
\theoremstyle{definition}
\newtheorem{definition}[theo]{Definition}
\theoremstyle{remark}
\newtheorem{rem}[theo]{Remark}
\numberwithin{equation}{section}

\newcommand{\R}{\mathbb{R}}
\newcommand{\N}{\mathbb{N}}
\newcommand{\M}{\mathbb{M}}
\newcommand{\divrg}{\textrm{div}\,}

\title{Size estimates for fat inclusions in an isotropic Reissner-Mindlin plate
\thanks{The first author is supported by PRIN $2015TTJN95$ ``Identification and
monitoring of complex structural systems". The second author is supported by 
FRA 2016 ``Problemi Inversi, dalla stabilit\`a alla ricostruzione'', Universit\`a degli Studi di Trieste. The second and the third authors are supported by Progetto GNAMPA 2017 ``Analisi di problemi inversi: stabilit\`a e ricostruzione'', Istituto Nazionale di Alta Matematica (INdAM).}}
\author{Antonino Morassi\thanks{Dipartimento Politecnico di Ingegneria e Architettura,
Universit\`a degli Studi di Udine, via Cotonificio 114, 33100
Udine, Italy. E-mail: \textsf{antonino.morassi@uniud.it}}, \  Edi
Rosset\thanks{Dipartimento di Matematica e Geoscienze,
Universit\`a degli Studi di Trieste, via Valerio 12/1, 34127
Trieste, Italy. E-mail: \textsf{rossedi@univ.trieste.it}} \ and
Sergio Vessella\thanks{Dipartimento di Matematica e Informatica ``Ulisse Dini'', Universit\`a degli Studi di Firenze,Viale Morgagni 67/a,
50134 Firenze, Italy. E-mail:
\textsf{sergio.vessella@unifi.it}}}

\begin{document}
\maketitle

\noindent \textbf{Abstract.} In this paper we consider the inverse
problem of determining, within an elastic isotropic thick plate
modelled by the Reissner-Mindlin theory, the possible presence of
an inclusion made of a different elastic material. Under some a
priori assumptions on the inclusion, we deduce constructive upper
and lower estimates of the area of the inclusion in terms of a
scalar quantity related to the work developed in deforming the
plate by applying simultaneously a couple field and a transverse
force field at the boundary of the plate. The approach allows to
consider plates with boundary of Lipschitz class.

\medskip

\medskip

\noindent \textbf{Mathematical Subject Classifications (2000):}
35R30, 35R25, 73C02.

\medskip

\medskip

\noindent \textbf{Key words:} inverse problems, elastic plates,
size estimates, unique continuation.

\section{Introduction} \label{sec:intro}

The inverse problem of damage identification via non-destructive
testing has attracted increasing interest in the applied and
mathematical literature of the last years. Its applicability is
particularly suited to those cases in which a simple visual
inspection of the damaged system is not sufficient to conclude
whether the defect is present or absent and, in the former case,
how extended it is. Non-destructive tests in dynamic regime are
rather common for large full-scale structures, such as bridges or
buildings. However, in case of simple structural elements such as
plates, the mechanical systems that will be considered in this
paper, static tests are easily executable and can provide valuable
information for solving the diagnostic problem.

In most of applications on plates, an accurate model describing
the structural defect, such as diffuse cracking in reinforced
concrete plates or yielding phenomena in metallic plates, is not a
priori available. Therefore, the defected plate is usually
modelled by introducing a variation of the elastic properties of
the material in an unknown subregion $D$ (\textit{inclusion}) of
the mid-surface $\Omega$ of the plate. Under the assumption that
the reference undamaged configuration of the plate is known, the
inverse problem is reduced to the determination of the inclusion
$D$ by comparing the results of boundary static tests executed on
the reference specimen (with $D= \varnothing$) and on the possibly
defected plate.

This appears to be a difficult inverse problem and a
general uniqueness result has not been obtained yet. Partial
answers have been given in the last ten years for thin elastic
plates described by the Kirchhoff-Love theory by pursuing a
relative modest, but realistic goal: to estimate the \textit{area}
of the unknown inclusion $D$ {}from a single static experiment.
More precisely, it was supposed to apply a given couple field
$\widehat{M}$ at the boundary $\partial \Omega$ of the plate in
the reference and in a possibly defected state, and to evaluate
the work $W_0$, $W$ exerted in deforming the undamaged and
defected specimen, respectively. Constructive estimates, {}from
above and {}from below, of $area(D)$ in terms of the difference
$|W_0 -W|$ were determined for Kirchhoff-Love elastic plates when
the background material is isotropic \cite{MRV07-IUMJ} or belongs
to a suitable class of anisotropy \cite{DiCLMRVW13-IP}. Extensions
to the limit cases of rigid inclusions and cavities were also
established \cite{MRV13-JIIPP}. Analogous results were derived for
size estimates of inclusions in shell structures (i.e., curved
Kirchhoff-Love plates) \cite{DiCLW13-ASNSP}, \cite{DiCLVW13-SIMA}.
For the sake of completeness we recall that the size estimates
approach traces back to the paper by Friedman \cite{Fri87-SIPA}
where, assuming that the measure of the possible inclusion in a
conducting body is a-priori known, a criterion was given to decide
{}from a single boundary measurement of current and corresponding
voltage whether the inclusion is present of not. Subsequently, the
method has been developed in \cite{Al-Ro}, \cite{KSS97-SIMA} and
\cite{ARS00-PAMS}, and extended also to the detection of
inclusions in elastic bodies \cite{Ik98-JIIPP}, \cite{Al-Mo-Ro02}.
Finally, we mention an interesting approach to size estimates
developed in \cite{KaKiMi12}, \cite{KaMi13} and in \cite{MiNg12}
where the translation method and the splitting method were
introduced, respectively.

All the available size estimates results for plate-like systems
have been obtained using the Kirchhoff-Love mechanical model of
plate, that is assuming that the material fibre initially
orthogonal to the mid-surface of the plate remains straight and
perpendicular to the mid-surface during deformation. Experiments
and numerical simulations show that this mechanical model
accurately describes the behavior of thin plates, whereas it
definitely looses precision as the thickness of the plate
increases. Specifically, when the thickness reaches the order of
one tenth the planar dimensions, the plates should be described by
means of an extension of the Kirchhoff-Love model, namely the
Reissner-Mindlin model \cite{Rei45}, \cite{Min51}, that takes into
account also the shear deformations through the thickness of the
plate. Moreover, it should be recalled that size estimates for the
Kirchhoff-Love plate model were derived under the a priori
condition that the mid-surface $\Omega$ is highly regular. This
technical assumption obstructs, for example, the application of
the size estimates to rectangular plates, in spite of their
frequent use in practical applications. In this paper, both the
two above mentioned limitations of the existing theory are
removed, and the size estimates approach is extended to the
Reissner-Mindlin model of plates with boundary $\partial \Omega$
of Lispchitz class.

Let us formulate our problem in mathematical terms. Let $D$, $D
\subset \subset \Omega$, be the subdomain of the mid-surface
$\Omega$ occupied by the inclusion, and denote by $h$ the constant
thickness of the plate. A transverse force field $\overline{Q}$
and a couple field $\overline{M}$ are supposed to be acting at the
boundary $\partial \Omega$ of the plate. Working in the framework
of the Reissner-Mindlin theory, at any point $x=(x_1,x_2) \in
\Omega$, we denote by $w=w(x)$ and by $\omega_\alpha =
\omega_\alpha (x)$, $\alpha =1,2$, the infinitesimal transverse
displacement at $x$ and the infinitesimal rotation of the
transverse material fibre through $x$, respectively. The pair
$(\varphi,w)$, with $\varphi_1=\omega_2$, $\varphi_2= -\omega_1$,
satisfies the Neumann boundary value problem
\begin{center}
\( {\displaystyle \!\!\!\!\!\!\!\!\!\!\!\!\left\{
\begin{array}{lr}
     \mathrm{\divrg}( (\chi_{\Omega\setminus D}S + \chi_D \widetilde{S}) (\varphi+\nabla
     w))=0,
      & \mathrm{in}\ \Omega,
        \vspace{0.25em}\\
      \mathrm{\divrg}( (\chi_{\Omega\setminus D} {\mathbb P} + \chi_D  \widetilde{{\mathbb P}} ) \nabla \varphi)-
      (\chi_{\Omega\setminus D}S + \chi_D \widetilde{S})(\varphi+\nabla w)=0, & \mathrm{in}\ \Omega,
          \vspace{0.25em}\\
      (S(\varphi+\nabla w))\cdot n= \overline{Q},
      & \mathrm{on}\ \partial \Omega,
        \vspace{0.25em}\\
      ({\mathbb P}\nabla \varphi) n = \overline{M}, &\mathrm{on}\ \partial
      \Omega,
          \vspace{0.25em}\\
\end{array}
\right. } \) \vskip -7.2em
\begin{eqnarray}
& & \label{eq:I-6-1}\\
& & \label{eq:I-6-2}\\
& & \label{eq:I-6-3}\\
& & \label{eq:I-6-4}
\end{eqnarray}
\end{center}
where $\chi_A$ denotes the characteristic function of the set $A$
and $n$ is the unit outer normal to $\partial \Omega$. In the
above equations, $(S,{\mathbb P})$ and
$(\widetilde{S},\widetilde{{\mathbb P}})$ are the second-order
shearing tensor and the fourth-order bending tensor of the
reference and defected plate, respectively.  The work exerted by
the boundary loads $(\overline{Q}, \overline{M})$ is denoted by
\begin{equation}
  \label{eq:I-7-1}
  W = \int_{\partial \Omega} \overline{Q}w + \overline{M}\cdot
  \varphi.
\end{equation}
When the inclusion $D$ is absent, the equilibrium problem
\eqref{eq:I-6-1}--\eqref{eq:I-6-4} becomes
\begin{center}
\( {\displaystyle \left\{
\begin{array}{lr}
     \mathrm{\divrg}( S(\varphi_0+\nabla w_0))=0,
      & \mathrm{in}\ \Omega,
        \vspace{0.25em}\\
      \mathrm{\divrg}( {\mathbb P}  \nabla \varphi_0)-S(\varphi_0+\nabla w_0)=0, & \mathrm{in}\ \Omega,
          \vspace{0.25em}\\
      (S(\varphi_0+\nabla w_0))\cdot n= \overline{Q},
      & \mathrm{on}\ \partial \Omega,
        \vspace{0.25em}\\
      ({\mathbb P}\nabla \varphi_0) n = \overline{M}, &\mathrm{on}\ \partial
      \Omega,
          \vspace{0.25em}\\
\end{array}
\right. } \) \vskip -7.2em
\begin{eqnarray}
& & \label{eq:I-7-2}\\
& & \label{eq:I-7-3}\\
& & \label{eq:I-7-4}\\
& & \label{eq:I-7-5}
\end{eqnarray}
\end{center}
where $(\varphi_0, w_0)$ is the deformation of the reference
plate. The corresponding work exerted by the boundary loads is
given by
\begin{equation}
  \label{eq:I-7-6}
  W_0 = \int_{\partial \Omega} \overline{Q}w_0 + \overline{M}\cdot
  \varphi_0.
\end{equation}
The first step towards the determination of the size estimates of
the area of the inclusion consists in proving that the strain
energy of the reference plate stored in the region $D$ is
comparable with the difference between the works exerted by the
boundary load fields in deforming the plate with and without the
inclusion. Under suitable assumptions on the jumps
$(\widetilde{{\mathbb P}} - {\mathbb P})$ and $(\widetilde{S}-S)$
of the elastic coefficients between the defected region $D$ and
the surrounding background material, and using the ellipticity of
the tensors $S$ and ${\mathbb P}$, the above property can be
stated as
\begin{equation}
  \label{eq:I-8-1}
  K_1 \int_D |\widehat{\nabla} \varphi_0|^2 + |\varphi_0 + \nabla
  w_0|^2 \leq
  |W-W_0|
  \leq
     K_2 \int_D |\widehat{\nabla} \varphi_0|^2 + |\varphi_0 + \nabla
  w_0|^2,
\end{equation}
for suitable positive constants $K_1$, $K_2$ only depending on the
data. Here, $\widehat{\nabla} \varphi_0  = \frac{1}{2} (\nabla
\varphi_0 + (\nabla \varphi_0)^T)$. We refer to Lemma
\ref{lem:energy} for the precise statement.

The lower bound for $area(D)$ follows {}from the right hand side
of \eqref{eq:I-8-1} and {}from regularity estimates for
the solution $(\varphi_0, w_0)$ to
\eqref{eq:I-7-2}--\eqref{eq:I-7-5}. It should be noticed that such
regularity estimates hold true also for anisotropic background
material, provided that the tensors ${\mathbb P}$ and $S$ have
suitable regularity.

In order to obtain the upper bound for $area(D)$, an estimate
{}from below of the strain energy expression appearing on the left
hand side of \eqref{eq:I-8-1} is needed. This issue is rather
technical and involves the determination of quantitative estimates
of unique continuation for the strain energy of the solution
$(\varphi_0, w_0)$ to the reference plate problem.

In this paper we assume that the inclusion $D$ satisfies the
\textit{fatness condition}
\begin{equation}
  \label{eq:I-9-1}
  area \left (
  \{
  x \in D \ | \ dist(x, \partial D) > h_1 \} \right )
  \geq \frac{1}{2} area(D),
\end{equation}
for a given positive number $h_1$. Under the assumption of
isotropic material, and requiring suitable regularity of the
tensors ${\mathbb P}$ and $S$, we shall prove a three spheres
inequality for the strain energy density $(|\widehat{\nabla}
\varphi_0|^2 + |\varphi_0 + \nabla w_0|^2)$ of the solution
$(\varphi_0, w_0)$ to \eqref{eq:I-6-1}--\eqref{eq:I-6-4}, see
Theorem \ref{theo:three-sphere_energy}. This three spheres
inequality for the energy strongly relies on a three spheres
inequality for $(|\varphi_0|^2 + |w_0|^2)$, with optimal exponent,
and on a generalized Korn inequality, both derived in
\cite{MRV16}. Our main result (see Theorem \ref{theo:D-fat})
states that if, for a given $h_1
>0$, the fatness-condition \eqref{eq:I-9-1} holds, and some a
priori assumptions on the unknown inclusion are satisfied, then
\begin{equation}
  \label{eq:I-10-1}
  C_1 \left | \frac{W-W_0}{W_0} \right | \leq area(D) \leq    C_2 \left | \frac{W-W_0}{W_0} \right
  |,
\end{equation}
where the constants $C_1$, $C_2$ only depend on the a priori data.
Clearly, the lower bound for $area(D)$ in \eqref{eq:I-10-1} continues to hold
even if the inclusion $D$ does not satisfy the fatness condition
\eqref{eq:I-9-1}.

The paper is organized as follows. Section $2$ collects some
notation. The formulation of the inverse problem is provided in
Section $3$, together with our main result (Theorem
\ref{theo:D-fat}). Section $4$ contains quantitative estimates of
unique continuation in the form of three spheres inequality
(Theorem \ref{theo:three-sphere_energy}) and Lipschitz propagation
of smallness property (Theorem \ref{theo:LPS}) for the strain
energy density of solutions to the Neumann problem for the
reference plate. The proof of Theorem \ref{theo:D-fat} is
presented in Section $5$, whereas Section $6$ is devoted to the
proof of Theorem \ref{theo:three-sphere_energy}.

\section{Notation} \label{sec:notation}

Let $P=(x_1(P), x_2(P))$ be a point of $\R^2$.
We shall denote by $B_r(P)$ the disk in $\R^2$ of radius $r$ and
center $P$ and by $R_{a,b}(P)$ the rectangle
$R_{a,b}(P)=\{x=(x_1,x_2)\ |\ |x_1-x_1(P)|<a,\ |x_2-x_2(P)|<b \}$. To simplify the notation,
we shall denote $B_r=B_r(O)$, $R_{a,b}=R_{a,b}(O)$.

\begin{definition}
  \label{def:2.1} (${C}^{k,1}$ regularity)
Let $\Omega$ be a bounded domain in ${\R}^{2}$. Given $k\in\N$, we
say that a portion $\Sigma$ of $\partial \Omega$ is of
\textit{class ${C}^{k,1}$ with constants $\rho_{0}$, $M_{0}>0$},
if, for any $P \in \Sigma$, there exists a rigid transformation of
coordinates under which we have $P=O$ and
\begin{equation*}
  \Omega \cap R_{\rho_0,M_0\rho_0}=\{x=(x_1,x_2) \in R_{\rho_0,M_0\rho_0}\quad | \quad
x_{2}>\psi(x_1)
  \},
\end{equation*}
where $\psi$ is a ${C}^{k,1}$ function on
$\left(-\rho_0,\rho_0\right)$ satisfying
\begin{equation*}
\psi(0)=0,
\end{equation*}
\begin{equation*}
\psi' (0)=0, \quad \hbox {when } k \geq 1,
\end{equation*}
\begin{equation*}
\|\psi\|_{{C}^{k,1}\left(-\frac{\rho_0}{M_0},\frac{\rho_0}{M_0}\right)} \leq M_{0}\rho_{0}.
\end{equation*}

\medskip
\noindent When $k=0$ we also say that $\Sigma$ is of
\textit{Lipschitz class with constants $\rho_{0}$, $M_{0}$}.
\end{definition}
\begin{rem}
  \label{rem:2.1}
  We use the convention to normalize all norms in such a way that their
  terms are dimensionally homogeneous with the $L^\infty$ norm and coincide with the
  standard definition when the dimensional parameter equals one, see \cite{MRV07-IUMJ} for details.
\end{rem}
For any $t>0$ we denote
\begin{equation}
  \label{eq:2.int_env}
  \Omega_{t}=\{x \in \Omega \mid \hbox{dist}(x,\partial
  \Omega)>t
  \}.
\end{equation}
Given a bounded domain $\Omega$ in $\R^2$ such that $\partial
\Omega$ is of class $C^{k,1}$, with $k\geq 0$, we consider as
positive the orientation of the boundary induced by the outer unit
normal $n$ in the following sense. Given a point
$P\in\partial\Omega$, let us denote by $\tau=\tau(P)$ the unit
tangent at the boundary in $P$ obtained by applying to $n$ a
counterclockwise rotation of angle $\frac{\pi}{2}$, that is
$\tau=e_3 \times n$, where $\times$ denotes the vector product in
$\R^3$, $\{e_1, e_2\}$ is the canonical basis in $\R^2$ and
$e_3=e_1 \times e_2$.

We denote by $\mathbb{M}^2$ the space of $2 \times 2$ real valued
matrices and by ${\mathcal L} (X, Y)$ the space of bounded linear
operators between Banach spaces $X$ and $Y$.

For every $2 \times 2$ matrices $A$, $B$ and for every $\mathbb{L}
\in{\mathcal L} ({\mathbb{M}}^{2}, {\mathbb{M}}^{2})$, we use the
following notation:
\begin{equation}
  \label{eq:2.notation_1}
  ({\mathbb{L}}A)_{ij} = L_{ijkl}A_{kl},
\end{equation}
\begin{equation}
  \label{eq:2.notation_2}
  A \cdot B = A_{ij}B_{ij}, \quad |A|= (A \cdot A)^{\frac {1}
  {2}}, \quad tr(A)=A_{ii},
\end{equation}
\begin{equation}
  \label{eq:2.notation_3}
  (A^T)_{ij}=A_{ji}, \quad \widehat{A} = \frac{1}{2}(A+A^T).
\end{equation}
Notice that here and in the sequel summation over repeated indexes
is implied.

\section{The inverse problem} \label{sec:inverseproblem}

Let us consider a plate, with constant thickness $h$, represented
by a bounded domain $\Omega$ in $\R^2$ having boundary of
Lipschitz class, with constants $\rho_0$ and $M_0$, and
satisfying
\begin{equation}
  \label{eq:Omegabound}
    \hbox{diam}(\Omega)\leq M_1\rho_0,
\end{equation}
\begin{equation}
  \label{eq:ballinsideOmega}
    B_{s_0\rho_0}(x_0)\subset\Omega,
\end{equation}
for some $s_0>0$ and $x_0\in \Omega$. Moreover, we assume that for
$r<h_0 \rho_0$, where $h_0 >0$ only depends on $M_0$, the domain
\begin{equation}
  \label{eq:OmegarhoLip}
    \Omega_r \ \hbox{is of Lipschitz class with constants } \rho_0,
    M_0.
\end{equation}
Condition \eqref{eq:OmegarhoLip} has been introduced to simplify
the arguments. However, it should be noticed that it is a rather
natural assumption, for instance trivially satisfied for polygonal
plates.

The \textit{reference} plate is assumed to be made by linearly
elastic isotropic material with Lam\'{e} moduli $\lambda$ and
$\mu$ satisfying the ellipticity conditions
\begin{equation}
  \label{eq:Lame-ell}
    \mu(x)\geq \alpha_0, \quad 2\mu(x)+3\lambda(x)\geq \gamma_0,
    \quad \hbox{in } \overline{\Omega},
\end{equation}
for given positive constants $\alpha_0$, $\gamma_0$, and the
regularity condition
\begin{equation}
  \label{eq:Lame-reg}
    \|\lambda\|_{C^{0,1}(\overline{\Omega})}+\|\mu\|_{C^{0,1}(\overline{\Omega})}\leq
    \alpha_1,
\end{equation}
where $\alpha_1$ is a given constant. Therefore, the shearing and
bending plate tensors take the form
\begin{equation}
  \label{eq:shearing-tensor}
    SI_2, \quad S=h\mu, \quad S \in
    C^{0,1}(\overline{\Omega}),
\end{equation}
\begin{equation}
  \label{eq:bending-tensor}
    {\mathbb P} A = B\left[(1-\nu)\widehat{A}+\nu tr(A)I_2\right],
    \quad {\mathbb P} \in
    C^{0,1}(\overline{\Omega}),
\end{equation}
where $I_2$ is the two-dimensional unit matrix, $A$ denotes a $2
\times 2$ matrix and
\begin{equation}
  \label{eq:Bending}
    B=\frac{Eh^3}{12(1-\nu^2)},
\end{equation}
with
\begin{equation}
  \label{eq:Young-Poisson}
   E=\frac{\mu(2\mu+3\lambda)}{\mu+\lambda}, \quad
   \nu=\frac{\lambda}{2(\mu+\lambda)}.
\end{equation}
By \eqref{eq:Lame-ell} and \eqref{eq:Lame-reg}, we have
\begin{equation}
  \label{eq:convex-S-Lame}
    h \sigma_0 \leq S \leq h \sigma_1, \quad \hbox{in } \overline{\Omega},
\end{equation}
and
\begin{equation}
  \label{eq:convex-P-Lame}
     \frac{h^3}{12} \xi_0 | \widehat{A} |^2 \leq {\mathbb P}A \cdot A \leq \frac{h^3}{12} \xi_1 | \widehat{A} |^2, \quad \hbox{in } \overline{\Omega},
\end{equation}
for every $2\times 2$ matrix $A$, where
\begin{equation}
  \label{eq:constants_dependence}
     \sigma_0 =\alpha_0, \quad \sigma_1 =\alpha_1, \quad \xi_0=\min\{2\alpha_0, \gamma_0\}, \quad
        \xi_1=2\alpha_1.
\end{equation}
Moreover,
\begin{equation}
  \label{eq:lip_operatori}
   \|S\|_{ C^{0,1}(\overline{\Omega})} \leq h\alpha_1, \qquad
     \|{\mathbb P}\|_{C^{0,1}(\overline{\Omega})} \leq Ch^3,
\end{equation}
with $C>0$ only depending on $\alpha_0$, $\alpha_1$,
$\gamma_0$.

Let the plate be subject to a transverse force field
$\overline{Q}$ and a couple field $\overline{M}$ acting on the
boundary $\partial \Omega$, and such that
\begin{equation}
  \label{eq:compatibilita-carico}
     \int_{\partial \Omega} \overline{Q}=0, \quad  \int_{\partial
     \Omega}(\overline{Q}x -\overline{M})=0,
\end{equation}
\begin{equation}
  \label{eq:regolarita-carico}
    \overline{Q} \in H^{-\frac{1}{2}}(\partial \Omega), \quad \overline{M} \in H^{-\frac{1}{2}}(\partial \Omega,
    \R^2).
\end{equation}
Under the above assumptions, the static equilibrium of the
reference plate is described within the Reissner-Mindlin theory by
the following Neumann boundary value problem
\begin{center}
\( {\displaystyle \left\{
\begin{array}{lr}
     \mathrm{\divrg}(S(\varphi_0+\nabla w_0))=0
      & \mathrm{in}\ \Omega,
        \vspace{0.25em}\\
      \mathrm{\divrg}({\mathbb P}\nabla \varphi_0)-S(\varphi_0+\nabla w_0)=0, & \mathrm{in}\ \Omega,
          \vspace{0.25em}\\
      (S(\varphi_0+\nabla w_0))\cdot n= \overline{Q},
      & \mathrm{on}\ \partial \Omega,
        \vspace{0.25em}\\
      ({\mathbb P}\nabla \varphi_0) n = \overline{M}, &\mathrm{on}\ \partial
      \Omega.
          \vspace{0.25em}\\
\end{array}
\right. } \) \vskip -7.5em
\begin{eqnarray}
& & \label{eq:intro-1}\\
& & \label{eq:intro-2}\\
& & \label{eq:intro-3}\\
& & \label{eq:intro-4}
\end{eqnarray}
\end{center}
Concerning the well-posedness of the above problem, it was proved
in \cite{MRV16} (Proposition $5.2$) that the problem
\eqref{eq:intro-1}--\eqref{eq:intro-4} admits a weak solution
$(\varphi_0, w_0) \in H^1(\Omega, \R^2) \times H^1(\Omega)$, that
is
\textit{for every} $\psi\in H^1(\Omega, \R^2)$ \textit{and for
every} $v\in H^1(\Omega)$,
\begin{equation}
  \label{eq:weak_form}
    \int_\Omega {\mathbb P}\nabla \varphi_0\cdot \nabla \psi + \int_\Omega S(\varphi_0+\nabla w_0)\cdot
        (\psi+\nabla v)=\int_{\partial\Omega}\overline{Q} v + \overline{M}\cdot \psi.
\end{equation}
The solution $(\varphi_0, w_0)$ can be uniquely identified
provided it satisfies the normalization conditions
\begin{equation}
  \label{eq:normalizzazione}
    \int_\Omega \varphi_0 =0, \quad \int_\Omega w_0=0.
\end{equation}
For this normalized solution, the following stability estimate holds
\begin{equation}
  \label{eq:dir7}
    \|\varphi_0\|_{H^1(\Omega)} + \frac{1}{\rho_0}\|w_0\|_{H^1(\Omega)}
        \leq \frac{C}{\rho_0^2}\left(\|\overline{M}\|_{H^{-\frac{1}{2}}(\partial\Omega)}+\rho_0
        \|\overline{Q}\|_{H^{-\frac{1}{2}}(\partial\Omega)}\right),
\end{equation}
where the constant $C>0$ only depends on $M_0$, $M_1$, $s_0$,
$\alpha_0$, $\alpha_1$, $\gamma_0$ and $\frac{\rho_0}{h}$.

\begin{rem}
Existence, uniqueness and $H^1$-stability for the Neumann problem
\eqref{eq:intro-1}--\eqref{eq:intro-4} can be proved for generic
anisotropic linearly elastic material with bounded shearing and
bending plate tensors satisfying suitable ellipticity conditions,
see Proposition $5.2$ in \cite{MRV16} for details. In fact, the
additional hypotheses of isotropy and regularity we have required
on the elastic coefficients are needed to obtain the key
quantitative estimate of unique continuation of the solution
$(\varphi_0, w_0)$ in the form of the three spheres inequality
\eqref{eq:three-sphere-1}.
\end{rem}
The inclusion $D$ is assumed to be a measurable, possibly
disconnected subset of $\Omega$ satisfying
\begin{equation}
  \label{eq:incl-interna}
  \hbox{dist}(D, \partial \Omega) \geq d_{0}\rho_0,
\end{equation}
where $d_0$ is a positive constant. The shearing and bending
tensors of the plate with the inclusion are denoted by
$(\chi_{\Omega\setminus D} S + \chi_D \widetilde{S})$,
$(\chi_{\Omega\setminus D} {\mathbb P} + \chi_D {\mathbb
{\widetilde{P}}})$, where $\chi_D$ is the characteristic function
of $D$ and $\widetilde{{S}} \in L^\infty(\Omega, {\M}^{2})$,
$\mathbb {\widetilde{P}} \in L^\infty(\Omega, {\cal L} ({\M}^{2},
{\M}^{2}))$. Differently {}from the surrounding material, no
isotropy condition is introduced on the inclusion $D$, and the
tensors $\widetilde{{S}}$, $\widetilde{{\mathbb P}}$ are requested
to satisfy the following properties:
\begin{enumerate}[i)]
\item \textit{Minor and major symmetry conditions}
\begin{equation}
  \label{eq:anto-11.1}
    \widetilde{S}_{\alpha \beta} = \widetilde{S}_{\beta \alpha}, \quad \alpha, \beta =1,2,
    \ \
    \hbox{a.e. in }  \Omega,
\end{equation}
\begin{equation}
  \label{eq:anto-11.2}
    \widetilde{P}_{\alpha \beta \gamma \delta}= \widetilde{P}_{\beta \alpha \gamma \delta}= \widetilde{P}_{\alpha \beta \delta \gamma}=\widetilde{P}_{\gamma \delta \alpha \beta}, \quad \alpha, \beta, \gamma,
    \delta=1,2, \ \ \hbox{a.e. in } \Omega.
\end{equation}
\item \textit{Bounds on the jumps $\widetilde {S}-{S}$,
$\widetilde {\mathbb P}-{\mathbb P}$ and uniform strong convexity
for $\widetilde{S}$ and $\widetilde{{\mathbb P}}$}

Either there exist $\eta>0$ and $\delta>1$ such that
\begin{equation}
  \label{eq:4.jump+bis}
  \eta{S} \leq \widetilde {S}-{S} \leq (\delta-1){S},  \quad \hbox{a.e. in } \Omega,
  \end{equation}
\begin{equation}
  \label{eq:4.jump+}
  \eta{\mathbb P} \leq \widetilde {\mathbb P}-{\mathbb P} \leq (\delta-1){\mathbb
  P},  \quad \hbox{a.e. in } \Omega,
  \end{equation}
or there exist $\eta>0$ and $0<\delta<1$ such that
\begin{equation}
  \label{eq:4.jump-bis}
  -(1-\delta){S} \leq \widetilde {S}-S \leq -\eta S,  \quad \hbox{a.e. in
  }\Omega,
  \end{equation}
\begin{equation}
  \label{eq:4.jump-}
  -(1-\delta){\mathbb P} \leq \widetilde {\mathbb P}-{\mathbb P} \leq -\eta{\mathbb
  P},  \quad \hbox{a.e. in }\Omega.
  \end{equation}
\end{enumerate}
As a further a priori information, let ${\mathcal F}>0$ be the
following ratio of norms of the boundary data
\begin{equation}
  \label{eq:frequency}
    {\mathcal F} = \frac{\|\overline{M}\|_{H^{-1/2}(\partial \Omega)}+\rho_0
\|\overline{Q}\|_{H^{-1/2}(\partial \Omega)}}
{\|\overline{M}\|_{H^{-1}(\partial \Omega)}+\rho_0
\|\overline{Q}\|_{H^{-1}(\partial \Omega)}}.
\end{equation}
Under the above assumptions, the equilibrium problem for the plate
with the inclusion $D$ is as follows
\begin{center}
\( {\displaystyle \!\!\!\!\!\!\!\!\!\!\!\!\left\{
\begin{array}{lr}
     \mathrm{\divrg}( (\chi_{\Omega\setminus D}S + \chi_D \widetilde{S}) (\varphi+\nabla
     w))=0,
      & \mathrm{in}\ \Omega,
        \vspace{0.25em}\\
      \mathrm{\divrg}( (\chi_{\Omega\setminus D} {\mathbb P} + \chi_D  \widetilde{{\mathbb P}} ) \nabla \varphi)-
      (\chi_{\Omega\setminus D}S + \chi_D \widetilde{S})(\varphi+\nabla w)=0, & \mathrm{in}\ \Omega,
          \vspace{0.25em}\\
      (S(\varphi+\nabla w))\cdot n= \overline{Q},
      & \mathrm{on}\ \partial \Omega,
        \vspace{0.25em}\\
      ({\mathbb P}\nabla \varphi) n = \overline{M}, &\mathrm{on}\ \partial
      \Omega.
          \vspace{0.25em}\\
\end{array}
\right. } \) \vskip -7.5em
\begin{eqnarray}
& & \label{eq:pbm-with-1}\\
& & \label{eq:pbm-with-2}\\
& & \label{eq:pbm-with-3}\\
& & \label{eq:pbm-with-4}
\end{eqnarray}
\end{center}
Problem \eqref{eq:pbm-with-1}--\eqref{eq:pbm-with-4} has a unique
solution $(\varphi, w) \in H^1(\Omega, \R^2) \times H^1(\Omega)$
satisfying the normalization conditions
\eqref{eq:normalizzazione}.

Finally, we introduce the works exerted by the boundary loads when
the inclusion is present or absent, respectively:
\begin{equation}
  \label{eq:work-with-D}
  W = \int_{\partial \Omega} \overline{Q}w + \overline{M}\cdot
  \varphi,
\end{equation}
\begin{equation}
  \label{eq:work-without-D}
  W_0 = \int_{\partial \Omega} \overline{Q}w_0 + \overline{M}\cdot
  \varphi_0.
\end{equation}
Our main theorem is as follows.
\begin{theo}
  \label{theo:D-fat}
  Let $\Omega$ be a bounded domain in $\R^{2}$, such that $\partial \Omega$
is of $C^{0,1}$ class with constants $\rho_{0}, M_{0}$ and
satisfying \eqref{eq:Omegabound}--\eqref{eq:OmegarhoLip}. Let $D$
be a measurable subset of $\Omega$ satisfying
\eqref{eq:incl-interna} and
\begin{equation}
  \label{eq:D-fat}
  \left| D_{h_{1}\rho_0} \right| \geq \frac {1} {2} \left| D \right|,
  \end{equation}
for a given positive constant $h_{1}$. Let the reference plate be
made by linearly elastic isotropic material with Lam\'{e} moduli
$\lambda$, $\mu$ satisfying \eqref{eq:Lame-ell},
\eqref{eq:Lame-reg}, and denote by $S$, ${\mathbb P}$ the
corresponding shearing and bending tensors given in
\eqref{eq:shearing-tensor}, \eqref{eq:bending-tensor},
respectively. The shearing tensor $\widetilde{{S}} \in
L^\infty(\Omega,{\M}^{2})$ and the bending tensor
$\widetilde{{\mathbb P}} \in L^\infty(\Omega,{\cal L} ({\M}^{2},
{\M}^{2}))$ of the inclusion $D$ are assumed to satisfy the
symmetry conditions \eqref{eq:anto-11.1}, \eqref{eq:anto-11.2}.

If \eqref{eq:4.jump+bis} and \eqref{eq:4.jump+} hold, then we have
\begin{equation}
  \label{eq:4.size_estimate+}
  \frac {1} {\delta-1} C^{+}_{1}\rho_0^2
  \frac{W_0-W}{W_0}
  \leq
  |D|
  \leq
  \frac {\delta} {\eta} C^{+}_{2}\rho_0^2
  \frac{W_0-W}{W_0}
  .
  \end{equation}

If, conversely, \eqref{eq:4.jump-bis} and \eqref{eq:4.jump-} hold,
then we have
\begin{equation}
  \label{eq:4.size_estimate-}
  \frac {\delta} {1-\delta} C^{-}_{1}\rho_0^2
  \frac{W-W_0}{W_0}
  \leq
  |D|
  \leq
  \frac {1} {\eta} C^{-}_{2}\rho_0^2
  \frac{W-W_0}{W_0}
  ,
  \end{equation}
where $C^{+}_{1}$, $C^{-}_{1}$ only depend on $M_{0}$, $M_1$,
$s_0$, $\frac{\rho_0}{h}$, $d_{0}$, $\alpha_{0}$, $\alpha_{1}$,
$\gamma_{0}$, whereas $C^{+}_{2}$, $C^{-}_{2}$ only depend on
$M_{0}$, $M_1$, $s_0$, $\frac{\rho_0}{h}$, $\alpha_{0}$,
$\alpha_{1}$, $\gamma_{0}$, $h_{1}$ and ${\mathcal F}$.
\end{theo}
\begin{rem}
Let us highlight that the upper bounds in \eqref{eq:4.size_estimate+},
\eqref{eq:4.size_estimate-} hold without assuming condition
\eqref{eq:incl-interna}, that is the inclusion is allowed to touch
the boundary of $\Omega$. This will be clear {}from the proof of Theorem
\ref{theo:D-fat} given in Section \ref{sec:proof}.
\end{rem}

\section{Unique continuation estimates} \label{sec:UC}

The key quantitative estimate of unique continuation for the
Reissner-Mindlin reference plate is the following three spheres
inequality, which was obtained in \cite[Theorem 7.1]{MRV16}.

\begin{theo}
   \label{theo:three-sphere}
Under the assumptions made in Section \ref{sec:inverseproblem},
let $(\varphi_0, w_0) \in H^1(\Omega, \R^2) \times H^1(\Omega)$ be
the solution to problem \eqref{eq:intro-1}--\eqref{eq:intro-4}
normalized by conditions \eqref{eq:normalizzazione}. Let $\bar
x\in\Omega$ and $R_1>0$ be such that $B_{R_1}(\bar x)\subset
\Omega$. Then there exists $\theta\in (0,1)$, $\theta$ depending
on $\alpha_0,\alpha_1,\gamma_0, \frac{\rho_0}{h}$ only, such that
if $0<R_3<R_2<R_1$ and $\frac{R_3}{R_1}\leq \frac{R_2}{R_1}\leq
\theta$, then we have
\begin{equation}
  \label{eq:three-sphere-1}
     \int_{B_{R_2}(\bar x)} \left|V\right|^2\leq C\left(\int_{B_{R_3}(\bar x)} \left|V\right|^2\right)^{\tau}\left(\int_{B_{R_1}(\bar x)} \left|V\right|^2\right)^{1-\tau}
\end{equation}
where
\begin{equation}
  \label{eq:three-sphere-2}
   \left|V\right|^2=|\varphi_0|^2+\frac{1}{\rho^2_0}|w_0|^2,
\end{equation}
$\tau\in(0,1)$ depends on $\alpha_0,\alpha_1,\gamma_0,
\frac{R_3}{R_1}, \frac{R_2}{R_1}, \frac{\rho_0}{h}$ only and $C$
depends on $\alpha_0,\alpha_1,\gamma_0, \frac{R_2}{R_1},
\frac{\rho_0}{h}$ only.
\end{theo}
In order to obtain the size estimates we need an estimate
analogous to \eqref{eq:three-sphere-1} for the strain energy
density
\begin{equation}
  \label{eq:energy}
   E(\varphi_0,w_0) = \left(|\widehat{\nabla} \varphi_0|^2 +\frac{1}{\rho^2_0} |\varphi_0 + \nabla
  w_0|^2\right)^\frac{1}{2}.
\end{equation}
\begin{theo}
   \label{theo:three-sphere_energy}
Under the assumptions made in Section \ref{sec:inverseproblem},
let $(\varphi_0, w_0) \in H^1(\Omega, \R^2) \times H^1(\Omega)$ be
the solution to problem \eqref{eq:intro-1}--\eqref{eq:intro-4}
normalized by conditions \eqref{eq:normalizzazione}.
There exist $\theta\in (0,1)$, $\tau\in(0,1)$, $C>0$ only
depending on $\alpha_0$, $\alpha_1$, $\gamma_0$,
$\frac{\rho_0}{h}$, such that for every $\rho\in (0,\rho_0)$ and
for every $\bar x\in\Omega$ such that $dist(\bar
x,\partial\Omega)\geq \frac{7}{2\theta}\rho$, we have
\begin{equation}
  \label{eq:three-sphere-1_energy}
    \int_{B_{3\rho}(\bar x)} E^2(\varphi_0,w_0)\leq C\
\left ( \frac{\rho_0}{\rho} \right )^2 \left(\int_{B_{\rho}(\bar
x)}
E^2(\varphi_0,w_0)\right)^{\tau}\left(\int_{B_{\frac{7}{2\theta}\rho}(\bar
x)} E^2(\varphi_0,w_0)\right)^{1-\tau}.
\end{equation}
\end{theo}
The main tool used to derive inequality
\eqref{eq:three-sphere-1_energy} {}from inequality
\eqref{eq:three-sphere-1} is the following Korn's inequality of
constructive type, which was established in \cite[Theorem
4.3]{MRV16}.
\begin{theo} [Generalized second Korn inequality]
   \label{theo:Korn_gener2}
Let $\Omega$ be a bounded domain in $\R^2$, with boundary
of Lipschitz class with constants $\rho_0$, $M_0$, satisfying
\eqref{eq:Omegabound}, \eqref{eq:ballinsideOmega}.
There exists a positive constant $C$ only depending on $M_0$, $M_1$ and $s_0$, such that,
for every $\varphi\in H^1(\Omega,\R^2)$ and for every $w\in H^1(\Omega,\R)$,
\begin{equation}
  \label{eq:korn}
    \|\nabla \varphi\|_{L^2(\Omega)}\leq C\left(\|\widehat{\nabla} \varphi\|_{L^2(\Omega)}+\frac{1}{\rho_0}\|\varphi+\nabla w\|_{L^2(\Omega)}\right).
\end{equation}
\end{theo}
It is also convenient to recall the following Poincar\'e
inequalities.
\begin{prop} [Poincar\'e inequalities]
   \label{prop:Poinc}
Let $\Omega$ be a bounded domain in $\R^2$, with boundary
of Lipschitz class with constants $\rho_0$, $M_0$, satisfying
\eqref{eq:Omegabound}. There exists a positive constant $C_P$ only depending on $M_0$ and $M_1$, such that
for every $u\in H^1(\Omega,\R^n)$, $n=1,2$,
\begin{equation}
  \label{eq:poinc1}
    \|u-u_\Omega\|_{L^2(\Omega)}\leq C_P\rho_0\|\nabla u\|_{L^2(\Omega)},
\end{equation}
\begin{equation}
  \label{eq:poinc2}
    \|u-u_G\|_{H^1(\Omega)}\leq
        \left(1+\left(\frac{|\Omega|}{|G|}\right)^\frac{1}{2}\right)\sqrt{1+C_P^2}\ \rho_0\|\nabla
        u\|_{L^2(\Omega)},
\end{equation}
where $G$, $G \subseteq \Omega$, is any measurable subset of
$\Omega$ with positive measure and $u_G= \frac{1}{|G|}\int_G u$.
\end{prop}
We refer to \cite[Example 3.5]{A-M-R08} and also \cite{A-M-R02}
for a quantitative evaluation of the constant $C_P$.
\begin{proof}[Proof of Theorem \ref{theo:three-sphere_energy}]
Let us apply Theorem \ref{theo:three-sphere} to the solution
$(\varphi^*,w^*)$ to \eqref{eq:intro-1}--\eqref{eq:intro-4}, where
\begin{equation}
  \label{eq:3sfere1}
   \varphi^* = \varphi_0-c_\rho, \quad w^* = w_0+c_\rho\cdot(x-\bar x)-d_\rho,
\end{equation}
with
\begin{equation}
  \label{eq:3sfere2}
   c_\rho = \frac{1}{|B_\rho|}\int_{B_\rho(\bar x)}\varphi_0, \quad
    d_\rho = \frac{1}{|B_\rho|}\int_{B_\rho(\bar x)}w_0,
\end{equation}
\begin{equation}
  \label{eq:3sfere3}
   R_3 =\rho, \quad R_2 =\frac{7}{2}\rho, \quad R_1 =\frac{7}{2\theta}\rho.
\end{equation}
Since $\varphi^* + \nabla w^* = \varphi_0 + \nabla w_0$ and $\nabla
\varphi^* = \nabla \varphi_0$, we have
\begin{equation}
  \label{eq:3sfere4}
     \int_{B_{\frac{7}{2}\rho}(\bar x)} |\varphi^*|^2+\frac{1}{\rho_0^2}|w^*|^2
        \leq C\left(\int_{B_{\rho}(\bar x)} |\varphi^*|^2+\frac{1}{\rho_0^2}|w^*|^2\right)^{\tau}\left(\int_{B_{\frac{7}{2\theta}\rho}(\bar x)} |\varphi^*|^2+\frac{1}{\rho_0^2}|w^*|^2\right)^{1-\tau},
\end{equation}
where $\tau\in (0,1)$, $C>0$ only depend on $\alpha_0$, $\alpha_1$, $\gamma_0$ and
$\frac{\rho_0}{h}$.

By applying Poincar\'e inequality \eqref{eq:poinc1} to the
functions $w^*$ and $\varphi^*$ and Korn inequality
\eqref{eq:korn} to $\varphi^*$ in the domain $B_{\rho}(\bar x)$
where these functions have zero mean value, we have
\begin{multline}
  \label{eq:3sfere5}
     \int_{B_\rho(\bar x)} |\varphi^*|^2+\frac{1}{\rho_0^2}|w^*|^2
        \leq C \int_{B_\rho(\bar x)} |\varphi^*|^2+\frac{\rho^2}{\rho_0^2}|\nabla w_0+c_\rho|^2
        \leq \\
        \leq C \int_{B_\rho(\bar x)} |\varphi^*|^2+\frac{\rho^2}{\rho_0^2}|\nabla w_0+\varphi_0|^2  +\frac{\rho^2}{\rho_0^2}|\varphi_0-c_\rho|^2\leq  \\
        \leq C\int_{B_\rho(\bar x)} \left(1+\frac{\rho^2}{\rho_0^2}\right)|\varphi^*|^2+\frac{\rho^2}{\rho_0^2}|\nabla w_0+\varphi_0|^2 \leq \\
        \leq C \int_{B_\rho(\bar x)} \rho^2\left(1+\frac{\rho^2}{\rho_0^2}\right)|\nabla\varphi^*|^2+\frac{\rho^2}{\rho_0^2}|\nabla w_0+\varphi_0|^2\leq \\
        \leq C \int_{B_\rho(\bar x)} \rho^2\left(1+\frac{\rho^2}{\rho_0^2}\right)|\widehat{\nabla}\varphi_0|^2+\left(1+\frac{\rho^2}{\rho_0^2}\right)|\nabla w_0+\varphi_0|^2\leq \\
        \leq C \rho_0^2\left(\int_{B_\rho(\bar x)}|\widehat{\nabla}\varphi_0|^2+\frac{1}{\rho_0^2}|\nabla w_0+\varphi_0|^2\right),
\end{multline}
with $C$ an absolute constant.

Similarly, we can estimate  the integral over
$B_{\frac{7}{2\theta}\rho}(\bar x)$ by using Poincar\'e inequality
\eqref{eq:poinc2} with $G=B_{\rho}(\bar x)$, $\Omega =
B_{\frac{7}{2\theta}\rho}(\bar x)$, obtaining
\begin{multline}
  \label{eq:3sfere6}
     \int_{B_{\frac{7}{2\theta}\rho}(\bar x)} |\varphi^*|^2+\frac{1}{\rho_0^2}|w^*|^2
        \leq C \rho_0^2\left(\int_{B_{\frac{7}{2\theta}\rho}(\bar x)}|\widehat{\nabla}\varphi_0|^2+\frac{1}{\rho_0^2}|\nabla w_0+\varphi_0|^2\right),
\end{multline}
with $C>0$ only depending on $\alpha_0$, $\alpha_1$, $\gamma_0$, $\frac{\rho_0}{h}$.

Next we need a Caccioppoli type inequality. To this aim, let us
consider a function $\eta\in C^\infty_0(\R^2)$, having compact
support contained in $B_{\frac{7}{2}\rho}(\bar x)$, satisfying
$\eta\equiv 1$ in $B_{3\rho}(\bar x)$, $\eta\geq 0$, $|\nabla
\eta|\leq \frac{C}{\rho}$, $C>0$ being an absolute constant.
Inserting in the weak formulation \eqref{eq:weak_form} the test
functions $\psi =\eta^2\varphi^*$, $v=\eta^2w^*$, we have
\begin{multline}
  \label{eq:3sfere7}
     \int_{B_{\frac{7}{2}\rho}(\bar x)} \eta^2{\mathbb P}\widehat{\nabla}\varphi^*
        \cdot \widehat{\nabla}\varphi^*+S(\varphi^*+\nabla w^*)\cdot (\eta^2(\varphi^*+\nabla w^*))\leq\\
        \leq C
        \int_{B_{\frac{7}{2}\rho}(\bar x)}\left(h^\frac{3}{2}|\nabla \eta||\varphi^*|\right)\left(h^\frac{3}{2}\eta|\widehat{\nabla} \varphi^*|\right)+
        \left(h^\frac{1}{2}\eta|\varphi^*+\nabla w^*|\right)\left(h^\frac{1}{2}|\nabla \eta||w^*|\right),
\end{multline}
where $C$ is an absolute constant.

By applying the ellipticity assumptions \eqref{eq:convex-S-Lame},
\eqref{eq:convex-P-Lame} and by using the standard inequality
$2ab\leq \epsilon a^2 + \frac{b^2}{\epsilon}$, $\epsilon >0$, we
have
\begin{multline}
  \label{eq:3sfere8}
     \int_{B_{\frac{7}{2}\rho}(\bar x)} \eta^2h^3|\widehat{\nabla}\varphi^*|^2
        +h\eta^2|\varphi^*+\nabla w^*|^2\leq\\
        \leq C\epsilon \int_{B_{\frac{7}{2}\rho}(\bar x)} \eta^2h^3|\widehat{\nabla}\varphi^*|^2
        +h\eta^2|\varphi^*+\nabla w^*|^2
        +\frac{C}{\epsilon\rho^2}\int_{B_{\frac{7}{2}\rho}(\bar x)}
        h^3|\varphi^*|^2 +h|w^*|^2,
\end{multline}
with $C$ only depending on $\alpha_0$, $\alpha_1$, $\gamma_0$. For
a suitable value of $\epsilon$, only depending on $\alpha_0$,
$\alpha_1$, $\gamma_0$, we have
\begin{multline}
  \label{eq:3sfere9}
     \int_{B_{3\rho}(\bar x)} |\widehat{\nabla}\varphi_0|^2
        +\frac{1}{\rho_0^2}|\varphi_0+\nabla w_0|^2
        \leq\frac{C}{\rho^2}\int_{B_{\frac{7}{2}\rho}(\bar x)}
        |\varphi^*|^2 +\frac{1}{\rho_0^2}|w^*|^2,
\end{multline}
with $C$ only depending on $\alpha_0$, $\alpha_1$, $\gamma_0$ and
$\frac{\rho_0}{h}$. By \eqref{eq:3sfere4}, \eqref{eq:3sfere5},
\eqref{eq:3sfere6}, \eqref{eq:3sfere9}, the thesis follows.
\end{proof}

Finally, the last mathematical tool of quantitative unique
continuation is the following result, whose proof is deferred in
Section \ref{sec:proofLPS}.

\begin{theo} [Lipschitz propagation of smallness]
   \label{theo:LPS}
Under the assumptions made in Section \ref{sec:inverseproblem},
for every $\rho>0$ and for every
$x\in\Omega_{\frac{7}{2\theta}\rho}$, we have
\begin{equation}
  \label{eq:LPS}
     \int_{B_{\rho}(x)} E^2(\varphi_0, w_0)\geq C_\rho\int_{\Omega} E^2(\varphi_0, w_0),
\end{equation}
where $C_\rho$ only depends on $\alpha_0$, $\alpha_1$, $\gamma_0$,
$\frac{\rho_0}{h}$, $M_0$, $M_1$, $s_0$, ${\mathcal F}$ and
$\rho$, and $\theta\in (0,1)$ has been introduced in Theorem
\ref{theo:three-sphere}, $\theta$ depending on
$\alpha_0,\alpha_1,\gamma_0, \frac{\rho_0}{h}$ only.
\end{theo}

\section{Proof of Theorem \ref{theo:D-fat}} \label{sec:proof}

The basic result connecting the presence of an inclusion to the
difference of the works corresponding to problems
\eqref{eq:intro-1}--\eqref{eq:intro-4} and
\eqref{eq:pbm-with-1}--\eqref{eq:pbm-with-4} is the following
Lemma.
\begin{lem} [Energy Lemma]
  \label{lem:energy}
  Let $\Omega$ be a bounded domain in $\R^2$ with boundary of Lipschitz class. Let
    $S, \widetilde{S} \in L^\infty(\Omega,   {\M}^{2})$ satisfy \eqref{eq:anto-11.1} and
  $\mathbb P, \widetilde{{\mathbb P}}\in L^\infty(\Omega,   {\cal L} ({\M}^{2}, {\M}^{2}))$
  satisfy \eqref{eq:anto-11.2}. Let us assume that the jumps
$(\widetilde{S} - S)$ and $(\widetilde{{\mathbb P}} - {\mathbb
P})$ satisfy either \eqref{eq:4.jump+bis}--\eqref{eq:4.jump+} or
\eqref{eq:4.jump-bis}--\eqref{eq:4.jump-}. Let $(\varphi_0,w_0)$,
$(\varphi,w)\in H^{1}(\Omega, {\Bbb R}^{2})\times H^{1}(\Omega)$
be the weak solutions to problems
\eqref{eq:intro-1}--\eqref{eq:intro-4},
\eqref{eq:pbm-with-1}--\eqref{eq:pbm-with-4}, respectively.

If \eqref{eq:4.jump+bis}--\eqref{eq:4.jump+} hold, then we have
\begin{multline}
  \label{eq:energy1}
  \frac {\eta} {\delta} \int_{D}\frac{h^3}{12}{\xi}_{0}|\widehat {\nabla} \varphi_{0}|^{2}
    +h{\sigma}_0|\varphi_0+\nabla w_0|^2
  \leq
  \int_{\partial \Omega} \overline{Q}(w_0-w) + \overline{M}\cdot
  (\varphi_0-\varphi)\leq \\
  \leq
  (\delta-1)\int_{D}\frac{h^3}{12}{\xi}_{1}|\widehat {\nabla} \varphi_{0}|^{2}
    +h{\sigma}_1|\varphi_0+\nabla w_0|^2
 .
\end{multline}
If \eqref{eq:4.jump-bis}--\eqref{eq:4.jump-} hold, then we have
\begin{multline}
  \label{eq:energy2}
  \eta \int_{D}\frac{h^3}{12}{\xi}_{0}|\widehat {\nabla} \varphi_{0}|^{2}
    +h{\sigma}_0|\varphi_0+\nabla w_0|^2
  \leq
  \int_{\partial \Omega} \overline{Q}(w_0-w) + \overline{M}\cdot
  (\varphi_0-\varphi)\leq \\
  \leq
  \frac{1-\delta}{\delta}\int_{D}\frac{h^3}{12}{\xi}_{1}|\widehat {\nabla} \varphi_{0}|^{2}
    +h{\sigma}_1|\varphi_0+\nabla w_0|^2
  .
\end{multline}
\end{lem}

\begin{proof}[Proof of Theorem \ref{theo:D-fat}]
Let us notice that by \eqref{eq:korn} and \eqref{eq:poinc1}, and
by the trivial estimate $\|\nabla w_0\|_{L^2(\Omega)}\leq
\|\varphi_0+\nabla w_0\|_{L^2(\Omega)} +
\|\varphi_0\|_{L^2(\Omega)}$, we have
\begin{multline}
  \label{eq:LPS10}
    \|\varphi_0\|_{H^1(\Omega)} + \frac{1}{\rho_0}\|w_0\|_{H^1(\Omega)}
        \leq \frac{C}\rho_0\left(
                \|\widehat{\nabla}\varphi_0\|_{L^2(\Omega)}+ \frac{1}{\rho_0}
                \|\varphi_0+\nabla w_0\|_{L^2(\Omega)} \right)
                \leq \\
                \leq C\left(\int_\Omega E^2(\varphi_0, w_0)\right)^{\frac{1}{2}},
\end{multline}
with $C$ only depending on $M_0$, $M_1$, $s_0$.

By standard regularity estimates for elliptic systems (see
\cite[Theorem 6.1]{Ca80}), by \eqref{eq:LPS10} and by the weak
formulation of the Neumann problem
\eqref{eq:intro-1}--\eqref{eq:intro-4}, we have
\begin{multline}
  \label{eq:proof_main1}
  \| \varphi_{0}\|_{{L}^{\infty}(D)} + \rho_0\|\widehat {\nabla} \varphi_{0}\|_{{L}^{\infty}(D)}
    + \|\nabla w_{0}\|_{{L}^{\infty}(D)}
  \leq
  C\left(\|\varphi_0\|_{H^1(\Omega)} + \frac{1}{\rho_0}\|w_0\|_{H^1(\Omega)}\right)
\leq \\
\leq
C\left(\int_\Omega E^2(\varphi_0,w_0)\right)^{\frac{1}{2}}
  \leq
  \frac{C}{\rho_0^{\frac{3}{2}}} \left ( \int_{\partial \Omega} \overline{Q}w_0 + \overline{M}\cdot
  \varphi_0
  \right )^{\frac {1}{2}},
\end{multline}
where the constant $C$ depends only on $M_0$, $M_1$, $s_0$,
$\alpha_{0}$, $\alpha_{1}$, $\gamma_{0}$, $\frac{\rho_0}{h}$,
$d_{0}$.

The lower bound for $|D|$ in \eqref{eq:4.size_estimate+},
\eqref{eq:4.size_estimate-} follows {}from the right hand side of
\eqref{eq:energy1}, \eqref{eq:energy2} and {}from
\eqref{eq:proof_main1}.

Now, let us prove the upper bound for $|D|$ in
\eqref{eq:4.size_estimate+}, \eqref{eq:4.size_estimate-}. Note
that
\begin{equation}
  \label{eq:proof_main2}
    \int_{D}\frac{h^3}{12}\xi_{0}|\widehat {\nabla} \varphi_{0}|^{2}
    +h\sigma_0|\varphi_0+\nabla w_0|^2\geq C\int_{D} E^2(\varphi_0,w_0),
\end{equation}
with $C$ only depending on $\alpha_0$, $\gamma_0$,
$\frac{\rho_0}{h}$.

Let us cover $D_{h_{1}}$ with internally non overlapping closed
squares $Q_{j}$ of side $l$, for $j=1,...,J$, with $l=\frac
{4\theta h_1}{2\sqrt 2\theta +7}$, where $\theta\in (0,1)$ is as
in Theorem \ref{theo:LPS}. By the choice of $l$ the squares
$Q_{j}$ are contained in $D$. Hence
\begin{equation}
  \label{eq:proof_main3}
  \int_{D} E^2(\varphi_0,w_0)
  \geq
  \int_{\bigcup_{j=1}^{J} Q_{j}} E^2(\varphi_0,w_0)
  \geq
  \frac {|D_{h_{1}}|}{l ^{n}}
  \int_{Q_{\bar {j}}} E^2(\varphi_0,w_0),
\end{equation}
where $\bar {j}$ is such that $\int_{Q_{\bar {j}}}
E^2(\varphi_0,w_0)=\min_{j}\int_{Q_{j}} E^2(\varphi_0,w_0)$. Let
$\bar {x}$ be the center of $Q_{\bar {j}}$. {}From
\eqref{eq:proof_main2}, \eqref{eq:proof_main3}, estimate
\eqref{eq:LPS} with $x=\bar {x}$ and $\rho=l/2$,
\eqref{eq:convex-S-Lame}, \eqref{eq:convex-P-Lame} and {from} the
weak formulation of \eqref{eq:intro-1}--\eqref{eq:intro-4} we have
\begin{equation}
  \label{eq:proof_main4}
  \int_{D}\frac{h^3}{12}\xi_{0}|\widehat {\nabla} \varphi_{0}|^{2}
    +h\sigma_0|\varphi_0+\nabla w_0|^2
  \geq
  K |D| W_0,
\end{equation}
where $K$ depends only on $\alpha_{0}$, $\alpha_{1}$,
$\gamma_{0}$, $M_{0}$, $M_1$, $s_0$, $\frac{\rho_0}{h}$, $h_{1}$
and ${\mathcal F}$.

The upper bound for $|D|$ in \eqref{eq:4.size_estimate+},
\eqref{eq:4.size_estimate-} follows {}from the left hand side of
\eqref{eq:energy1},\eqref{eq:energy2} and {}from
\eqref{eq:proof_main4}.
\end{proof}

\section{Proof of Theorem \ref{theo:LPS}} \label{sec:proofLPS}

Let us premise the following Proposition.
\begin{prop}
   \label{prop:dopo_dualità}
Let $\Omega$ be a bounded domain in $\R^2$, with boundary of
Lipschitz class with constants $\rho_0$, $M_0$, satisfying
\eqref{eq:Omegabound}.
Let $S \in
C^{0,1}(\overline{\Omega}, {\M}^{2})$ and $\mathbb P \in C^{0,1}
(\overline{\Omega}, {\cal L} ({\M}^{2}, {\M}^{2}))$ given by
\eqref{eq:shearing-tensor}, \eqref{eq:bending-tensor} with the Lam\'e moduli satisfying
\eqref{eq:Lame-ell},
\eqref{eq:Lame-reg}.
Let $\overline{M} \in H^{-
\frac{1}{2}} (\partial \Omega, \R^2)$ and $\overline{Q} \in H^{
-\frac{1}{2}} (\partial \Omega)$ satisfy the compatibility
conditions \eqref{eq:compatibilita-carico}. Let $(\varphi_0, w_0) \in
H^1(\Omega, \R^2) \times H^1(\Omega)$ be the solution of the
problem \eqref{eq:intro-1}--\eqref{eq:intro-4}, normalized by the
conditions \eqref{eq:normalizzazione}. Then there exists a
positive constant $C$ only depending on $M_0$, $M_1$,
$\alpha_0$, $\alpha_1$, $\gamma_0$, $\frac{\rho_0}{h}$, such
that
\begin{equation}
  \label{eq:dopo_dualità}
    \|\overline{M}\|_{H^{-1}(\partial \Omega,\R^2)}+\rho_0 \|\overline{Q}\|_{H^{-1}(\partial \Omega)}\leq C\rho_0^2
        \left(\|\varphi_0\|_{L^{2}(\partial \Omega,\R^2)}+\frac{1}{\rho_0} \|w_0\|_{L^{2}(\partial \Omega)}
        \right).
\end{equation}
\end{prop}
\begin{rem}
   \label{rem:no_isotropia}
    Let us highlight that the above Proposition, as well as Lemma \ref{lem:prima_dualità}, on which its proof is based, hold true for anisotropic materials.
\end{rem}
\begin{proof}[Proof of Theorem \ref{theo:LPS}]
By Proposition $5.5$ in \cite{ARRV} and by \eqref{eq:OmegarhoLip},
there exists $h_2 >0$ only depending on $M_0$ such that $\Omega_{
\frac{4}{\theta}\rho}$ is connected and of Lipschitz class with constant $\rho_0$, $M_0$, for every
$\rho \leq \frac{\theta}{4}h_2\rho_0$. Let $\rho \leq \frac{\theta}{4}h_2\rho_0$.

Given any point $y \in \Omega_{\frac{4}{\theta}\rho}$, let
$\gamma$ be an arc in $\Omega_{\frac{4}{\theta}\rho}$ joining $x$
and $y$. Let us define the points $\{x_i\}$, $i=1,...,L$, as
follows: $x_1=x$, $x_{i+1}=\gamma (t_i)$, where $t_i =\max \{ t \
\hbox{s.t.} \ |\gamma (t)-x_i|=2\rho\}$ if $|x_i -y| > 2\rho$,
otherwise let $i=L$ and stop the process. By construction, the
disks $B_\rho(x_i)$ are pairwise disjoint and
$|x_{i+1}-x_i|=2\rho$, $i=1,...,L-1$, $|x_L -y|\leq 2 \rho$.

By applying Theorem \ref{theo:three-sphere_energy} and denoting
$E(\varphi_0, w_0)=E$ to simplify the notation, we have
\begin{equation}
  \label{LPS-p1-e1}
    \int_{B_{\rho}( x_{i+1})} E^2
    \leq C\
    \left ( \frac{\rho_0}{\rho}\right )^2
    \left(\int_{B_{\rho}(x_i)} E^2 \right)^{\tau}
    \left(\int_{B_{\frac{7}{2\theta}\rho}(x_i)}
    E^2\right)^{1-\tau},
\end{equation}
for $i=1,...,L-1$, where $\tau \in (0,1)$ and $C>0$ only depend on
$\alpha_0$, $\alpha_1$, $\gamma_0$ and $\frac{\rho_0}{h}$.

Let us apply the Caccioppoli inequality \eqref{eq:3sfere9} to
estimate {}from above the second integral on the right hand side
of \eqref{LPS-p1-e1}, namely
\begin{equation}
  \label{LPS-p2-e1}
    \int_{B_{\frac{7}{2\theta}\rho}(x_i)}  E^2
        \leq
    \frac{C}{\rho^2} \int_{B_{\frac{4}{\theta}\rho}(x_i)}
        |\varphi_0^*|^2 +\frac{1}{\rho_0^2}|w_0^*|^2
    \leq
    \frac{C}{\rho^2}
    \int_{\Omega}
    |\varphi_0^*|^2 +\frac{1}{\rho_0^2}|w_0^*|^2,
\end{equation}
where $\varphi_0^*=\varphi_0 -c$, $w_0^* = w_0 +c \cdot (x-
\overline{x}) -d$, with $c \in \R^2$, $d \in \R$, $\overline{x}
\in \R^2$ to be chosen later, and where $C>0$ only depends on
$\alpha_0$, $\alpha_1$, $\gamma_0$ and $\frac{\rho_0}{h}$.

By \eqref{LPS-p1-e1} and \eqref{LPS-p2-e1}, and using an iteration
argument, we have
\begin{equation}
  \label{LPS-p2-e2}
    \frac
    {\rho^2  \int_{B_{\rho}(y)}  E^2  }
    { \int_{\Omega}
    |\varphi_0^*|^2 +\frac{1}{\rho_0^2}|w_0^*|^2  }
    \leq
    C \left ( \frac{\rho_0}{\rho}\right )^2
    \left (
    \frac
    {\rho^2  \int_{B_{\rho}(x)}  E^2  }
    {\int_{\Omega}
    |\varphi_0^*|^2 +\frac{1}{\rho_0^2}|w_0^*|^2 }
    \right )^{\tau^L},
\end{equation}
where, by \eqref{eq:Omegabound}, $L \leq C_1 \left (
\frac{\rho_0}{\rho}\right )^2$, with $C_1>0$ only depending on
$M_1$, and $C$ is as above.

Let us tessellate $ \Omega_{ \frac{5}{\theta}\rho}$ with internally non
overlapping closed squares of side $l = \frac{2\rho}{\sqrt{2}}$.
By \eqref{eq:Omegabound}, their number is dominated by $N=
\frac{|\Omega|}{2\rho^2}\leq C \left ( \frac{\rho_0}{\rho} \right
)^2$, with $C>0$ only depending on $M_1$. Then, by
\eqref{LPS-p2-e2} we have
\begin{equation}
  \label{LPS-p3-e1}
    \frac
    {\rho^2  \int_{  \Omega_{  \frac{5}{\theta}\rho  }  }  E^2  }
    { \int_{\Omega}
    |\varphi_0^*|^2 +\frac{1}{\rho_0^2}|w_0^*|^2  }
    \leq
    C \left ( \frac{\rho_0}{\rho}\right )^4
    \left (
    \frac
    {\rho^2  \int_{B_{\rho}(x)}  E^2  }
    {\int_{\Omega}
    |\varphi_0^*|^2 +\frac{1}{\rho_0^2}|w_0^*|^2 }
    \right )^{\tau^L},
\end{equation}
where $C>0$ only depends on $\alpha_0$, $\alpha_1$, $\gamma_0$, $
\frac{\rho_0}{h}$ and $M_1$.

In the next step, we shall estimate {}from below $\int_{  \Omega_{
\frac{5}{\theta}\rho  }  }  E^2 $. Let us choose
\begin{equation}
  \label{LPS-p3-e2}
    c= \frac{1}{| \Omega_{  \frac{5}{\theta}\rho  } | } \int_{\Omega_{  \frac{5}{\theta}\rho
    }} \varphi_0, \quad d= \frac{1}{| \Omega_{  \frac{5}{\theta}\rho  } | } \int_{\Omega_{  \frac{5}{\theta}\rho
    }} w_0, \quad \overline{x} = \frac{1}{| \Omega_{  \frac{5}{\theta}\rho  } | } \int_{\Omega_{  \frac{5}{\theta}\rho
    }} x
\end{equation}
and let $\rho \leq \frac{\theta}{5}h_2 \rho_0$, so that $\Omega_{
\frac{5}{\theta}\rho  }$ is connected and of Lipschitz class with
constants $\rho_0$, $M_0$. By using Korn inequality
\eqref{eq:korn} and Poincar\'{e} inequality \eqref{eq:poinc1} in
\eqref{LPS-p3-e1}, and recalling that
$E(\varphi_0^*,w_0^*)=E(\varphi_0,w_0)$, we have
\begin{equation}
  \label{LPS-p4-e1}
    C \left ( \frac{\rho}{\rho_0}\right )^6
    \
    \frac
    {\int_{  \Omega_{  \frac{5}{\theta}\rho  }  }  |\varphi_0^*|^2 + \frac{1}{\rho_0^2 }|w_0^*|^2  }
    { \int_{\Omega}
    |\varphi_0^*|^2 +\frac{1}{\rho_0^2}|w_0^*|^2  }
    \leq
    \left (
    \frac
    {\rho^2  \int_{B_{\rho}(x)}  E^2  }
    {\int_{\Omega}
    |\varphi_0^*|^2 +\frac{1}{\rho_0^2}|w_0^*|^2 }
    \right )^{\tau^L},
\end{equation}
where $C>0$ only depends on $M_0$, $M_1$, $s_0$, $\alpha_0$,
$\alpha_1$, $\gamma_0$ and $ \frac{\rho_0}{h}$.

Recalling that $\int_\Omega \varphi_0 =0$, $\int_\Omega w_0 =0$,
and since
\begin{equation}
  \label{LPS-p4-e1bis}
    |\Omega \setminus \Omega_{  \frac{5}{\theta}\rho  }| \leq
    C\rho\rho_0, \quad |\Omega_{  \frac{5}{\theta}\rho  }| \geq C
    \rho_0^2,
\end{equation}
with $C>0$ only depending on $M_0$ and $M_1$ (see \cite[Appendix]{Al-Ro} for
details), by H\"{o}lder inequality we have
\begin{equation}
  \label{LPS-p4-e2}
    |c| \leq \frac{C}{\rho_0} \left(\frac{\rho}{\rho_0}\right) ^{\frac{1}{2}}
    \left ( \int_{  \Omega \setminus \Omega_{  \frac{5}{\theta}\rho  }  }
    |\varphi_0|^2 \right )^{ \frac{1}{2}  },
\end{equation}
\begin{equation}
  \label{LPS-p4-e3}
    |d| \leq \frac{C}{\rho_0} \left(\frac{\rho}{\rho_0}\right) ^{\frac{1}{2}}
    \left ( \int_{  \Omega \setminus \Omega_{  \frac{5}{\theta}\rho  }  }
    |w_0|^2 \right )^{ \frac{1}{2}  }
\end{equation}
and, therefore,
\begin{equation}
  \label{LPS-p4-e4}
    \left|
    \left ( \int_{  \Omega \setminus \Omega_{  \frac{5}{\theta}\rho  }  }
    |\varphi_0^*|^2 \right )^{ \frac{1}{2}  } - \left ( \int_{  \Omega \setminus \Omega_{  \frac{5}{\theta}\rho  }  }
    |\varphi_0|^2 \right )^{ \frac{1}{2}  }
    \right|
    \leq
    C_1 \frac{\rho}{\rho_0} \left ( \int_{  \Omega \setminus \Omega_{  \frac{5}{\theta}\rho  }  }
    |\varphi_0|^2 \right )^{ \frac{1}{2}  },
\end{equation}
\begin{equation}
  \label{LPS-p4-e5}
    \left|
    \left ( \int_{  \Omega  }
    |\varphi_0^*|^2 \right )^{ \frac{1}{2}  } - \left ( \int_{  \Omega   }
    |\varphi_0|^2 \right )^{ \frac{1}{2}  }
    \right|
    \leq
    C_1 \left ( \frac{\rho}{\rho_0} \right )^{ \frac{1}{2}  } \left ( \int_{  \Omega \setminus \Omega_{  \frac{5}{\theta}\rho  } }
    |\varphi_0|^2 \right )^{ \frac{1}{2}  },
\end{equation}
where $C_1>0$ depends only on $M_0$ and $M_1$. Assuming, in
addition, $\rho \leq \min \{ \frac{1}{2C_1},
\frac{1}{4C_1^2}\}\rho_0$, {}from \eqref{LPS-p4-e4},
\eqref{LPS-p4-e5} we have
\begin{equation}
  \label{LPS-p5-e1}
    \int_{  \Omega \setminus \Omega_{  \frac{5}{\theta}\rho  }  }
    |\varphi_0^*|^2
    \leq
    \frac{9}{4} \int_{  \Omega \setminus \Omega_{  \frac{5}{\theta}\rho  }  }
    |\varphi_0|^2,
\end{equation}
\begin{equation}
  \label{LPS-p5-e2}
    \int_{  \Omega   }
    |\varphi_0^*|^2
    \geq
    \frac{1}{4}  \int_{  \Omega }
    |\varphi_0|^2.
\end{equation}
By \eqref{LPS-p4-e2}, \eqref{LPS-p4-e3} we can estimate
\begin{multline}
  \label{LPS-p5-e3}
    \left|
    \left ( \int_{  \Omega \setminus \Omega_{  \frac{5}{\theta}\rho  }  }
    |w_0^*|^2 \right )^{ \frac{1}{2}  } - \left ( \int_{  \Omega \setminus \Omega_{  \frac{5}{\theta}\rho  }  }
    |w_0|^2 \right )^{ \frac{1}{2}  }
    \right|
    \leq
    \left ( \int_{  \Omega \setminus \Omega_{  \frac{5}{\theta}\rho  }  }
    |c \cdot (x-\overline{x})+d|^2 \right )^{ \frac{1}{2}  } \leq
    \\
    \leq
    C (\rho_0 |c| + |d|) | \Omega \setminus \Omega_{  \frac{5}{\theta}\rho  }
    |^{  \frac{1}{2}  }
    \leq
    C_2 \rho
    \left (
    \left ( \int_{  \Omega \setminus \Omega_{  \frac{5}{\theta}\rho  }  }
    |\varphi_0|^2 \right )^{ \frac{1}{2}  } +
    \frac{1}{\rho_0}
    \left ( \int_{  \Omega \setminus \Omega_{  \frac{5}{\theta}\rho  }  }
    |w_0|^2 \right )^{ \frac{1}{2}  }
    \right )
\end{multline}
and, taking the squares, we obtain
\begin{equation}
  \label{LPS-p5-e4}
     \int_{  \Omega \setminus \Omega_{  \frac{5}{\theta}\rho  }  }
    |w_0^*|^2
    \leq
    \left (
    2 + 4 C_2^2 \left ( \frac{\rho}{\rho_0}\right )^2\right )
    \int_{  \Omega \setminus \Omega_{  \frac{5}{\theta}\rho  }  }
    |w_0|^2
    +
    4C_2^2 \rho^2 \int_{  \Omega \setminus \Omega_{  \frac{5}{\theta}\rho  }  }
    |\varphi_0|^2,
\end{equation}
where $C_2>0$ only depends on $M_0$ and $M_1$. {}From
\eqref{LPS-p5-e1} and \eqref{LPS-p5-e4}, and assuming also $\rho
\leq \frac{3}{4C_2}\rho_0$, we have
\begin{equation}
  \label{LPS-p6-e1}
    \int_{ \Omega \setminus \Omega_{  \frac{5}{\theta}\rho  }  }  |\varphi_0^*|^2 + \frac{1}{\rho_0^2 }|w_0^*|^2
    \leq
    \frac{9}{2}
    \int_{ \Omega \setminus \Omega_{  \frac{5}{\theta}\rho  }  }  |\varphi_0|^2 + \frac{1}{\rho_0^2
    }|w_0|^2.
\end{equation}
By repeating calculations similar to those performed in obtaining
\eqref{LPS-p5-e3}, we have
\begin{equation}
  \label{LPS-p5-e1bis}
    \left|
    \left ( \int_{  \Omega }
    |w_0^*|^2 \right )^{ \frac{1}{2}  } - \left ( \int_{  \Omega  }
    |w_0|^2 \right )^{ \frac{1}{2}  }
    \right|
    \leq
    C_3 (\rho \rho_0)^{ \frac{1}{2} }
    \left (
    \left ( \int_{  \Omega  }
    |\varphi_0|^2 \right )^{ \frac{1}{2}  } +
    \frac{1}{\rho_0}
    \left ( \int_{  \Omega  }
    |w_0|^2 \right )^{ \frac{1}{2}  }
    \right ),
\end{equation}
where $C_3>0$ only depends on $M_0$ and $M_1$. Taking the squares,
we deduce
\begin{equation}
  \label{LPS-p6-e2}
    \frac{1}{\rho_0^2}
    \int_{  \Omega }
    |w_0^*|^2
    \geq
     \frac{1}{\rho_0^2}
    \int_{  \Omega }
    |w_0|^2
    +
    \left (
    2C_3^2 \frac{\rho}{\rho_0}-2\sqrt{2}C_3 \left (
    \frac{\rho}{\rho_0}\right )^{  \frac{1}{2} }\right )
    \left ( \int_{  \Omega  }
    |\varphi_0|^2  +
    \frac{1}{\rho_0^2}
    |w_0|^2 \right ),
\end{equation}
where $C_3>0$ only depends on $M_0$ and $M_1$. By
\eqref{LPS-p5-e2} and \eqref{LPS-p6-e2}, and taking $\rho \leq
\frac{1}{2\cdot 16^2 C_3^2}\rho_0$, we have
\begin{equation}
  \label{LPS-p6-e3}
    \int_{\Omega}
    |\varphi_0^*|^2 +\frac{1}{\rho_0^2}|w_0^*|^2
    \geq
    \frac{1}{8}
    \int_{\Omega}
    |\varphi_0|^2 +\frac{1}{\rho_0^2}|w_0|^2.
\end{equation}
Let us rewrite the quotient appearing on the left hand side of
\eqref{LPS-p4-e1} as
\begin{equation}
  \label{LPS-p6-e4}
    \frac
    {\int_{  \Omega_{  \frac{5}{\theta}\rho  }  }  |\varphi_0^*|^2 + \frac{1}{\rho_0^2 }|w_0^*|^2  }
    { \int_{\Omega}
    |\varphi_0^*|^2 +\frac{1}{\rho_0^2}|w_0^*|^2  }
    = 1 -
     \frac
    {\int_{  \Omega \setminus \Omega_{  \frac{5}{\theta}\rho  }  }  |\varphi_0^*|^2 + \frac{1}{\rho_0^2 }|w_0^*|^2  }
    { \int_{\Omega}
    |\varphi_0^*|^2 +\frac{1}{\rho_0^2}|w_0^*|^2  }.
\end{equation}
By \eqref{LPS-p6-e1} and \eqref{LPS-p6-e3} we have
\begin{equation}
  \label{LPS-p7-e1}
    \frac
    {\int_{  \Omega \setminus \Omega_{  \frac{5}{\theta}\rho  }  }  |\varphi_0^*|^2 + \frac{1}{\rho_0^2 }|w_0^*|^2  }
    { \int_{\Omega}
    |\varphi_0^*|^2 +\frac{1}{\rho_0^2}|w_0^*|^2  }
    \leq 36 \
    \frac
    {\int_{  \Omega \setminus \Omega_{  \frac{5}{\theta}\rho  }  }  |\varphi_0|^2 + \frac{1}{\rho_0^2 }|w_0|^2  }
    { \int_{\Omega}
    |\varphi_0|^2 +\frac{1}{\rho_0^2}|w_0|^2  }.
\end{equation}
{}From H\"{o}lder's inequality, Sobolev embedding theorem and
\eqref{LPS-p4-e1bis} we have
\begin{equation}
  \label{LPS-p7-e2}
    \int_{  \Omega \setminus \Omega_{  \frac{5}{\theta}\rho  }  }  |\varphi_0|^2
    \leq C \rho^{1- \frac{2}{p}}\rho_0^{1+ \frac{2}{p}}
    \int_\Omega |\nabla \varphi_0|^2,
\end{equation}
\begin{equation}
  \label{LPS-p7-e3}
    \int_{  \Omega \setminus \Omega_{  \frac{5}{\theta}\rho  }  }  |w_0|^2
    \leq C \rho^{1- \frac{2}{p}}\rho_0^{1+ \frac{2}{p}}
    \int_\Omega |\nabla w_0|^2,
\end{equation}
with $C>0$ only depending on $M_0$ and $M_1$, and $p$ a given
number, $p >2$, for instance $p=3$. By \eqref{LPS-p7-e2} and \eqref{LPS-p7-e3}, we
have
\begin{equation}
  \label{LPS-p7-e4}
    \frac
    {\int_{  \Omega \setminus \Omega_{  \frac{5}{\theta}\rho  }  }  |\varphi_0|^2 + \frac{1}{\rho_0^2 }|w_0|^2  }
    { \int_{\Omega}
    |\varphi_0|^2 +\frac{1}{\rho_0^2}|w_0|^2  }
    \leq
    C \rho^{1- \frac{2}{p}}\rho_0^{1+ \frac{2}{p}}
    \frac
    {\int_{  \Omega  }  |\nabla \varphi_0|^2 + \frac{1}{\rho_0^2 }|\nabla w_0|^2  }
    { \int_{\Omega}
    |\varphi_0|^2 +\frac{1}{\rho_0^2}|w_0|^2  },
\end{equation}
with $C$ and $p$ as above.

Now, let us recall the following trace inequality (see
\cite[Theorem 1.5.1.10]{Gris85})
\begin{equation}
  \label{LPS-p8-e1}
    \int_{\partial \Omega} |w_0|^2 \leq
    C
    \left (
    \left (
    \int_{\Omega}
    |\nabla w_0|^2
    \right )^{\frac{1}{2}}
        \cdot
    \left (
    \int_{\Omega}
    | w_0|^2
    \right )^{\frac{1}{2}}
    +\frac{1}{\rho_0}
    \int_{\Omega}
    | w_0|^2
    \right ),
\end{equation}
with $C$ only depending on $M_0$ and $M_1$. Therefore, by
\eqref{LPS-p8-e1} and Poincar\'{e} inequality \eqref{eq:poinc1},
\begin{equation}
  \label{LPS-p8-e2}
    \int_{\partial \Omega} |w_0|^2 \leq
    C\rho_0
    \left (
    \int_{\Omega}
    | w_0|^2
    \right )^{\frac{1}{2}}
    \left (
    \int_{\Omega}
    |\nabla \varphi_0|^2
    +
    \frac{1}{\rho_0^2}
    |\nabla w_0|^2
    \right )^{\frac{1}{2}},
\end{equation}
where $C>0$ only depends on $M_0$ and $M_1$. Similarly, by a trace
inequality analogous to \eqref{LPS-p8-e1} and by Poincar\'{e}
inequality \eqref{eq:poinc1}, we have
\begin{equation}
  \label{LPS-p8-e3}
    \int_{\partial \Omega} |\varphi_0|^2 \leq
    C
    \left (
    \int_{\Omega}
    | \varphi_0|^2
    \right )^{\frac{1}{2}}
    \left (
    \int_{\Omega}
    |\nabla \varphi_0|^2
    +
    \frac{1}{\rho_0^2}
    |\nabla w_0|^2
    \right )^{\frac{1}{2}},
\end{equation}
with $C>0$ only depending on $M_0$ and $M_1$. Therefore, by
\eqref{LPS-p8-e2} and \eqref{LPS-p8-e3} we have
\begin{equation}
  \label{LPS-p8-e4}
    \int_{\Omega}
    |\varphi_0|^2
    +
    \frac{1}{\rho_0^2}
    |w_0|^2
    \geq
    C
    \
    \frac
    { \left ( \int_{  \partial \Omega  }  |\varphi_0|^2 \right )^2 + \frac{1}{\rho_0^4 } \left ( \int_{  \partial \Omega  } |w_0|^2 \right )^2 }
    { \int_{\Omega}
    |\nabla \varphi_0|^2 +\frac{1}{\rho_0^2}|\nabla w_0|^2  }.
\end{equation}
with $C>0$ only depending on $M_0$ and $M_1$. {}From  \eqref{eq:dir7},
\eqref{eq:dopo_dualità} and \eqref{LPS-p8-e4},
we deduce
\begin{equation}
  \label{LPS-p9-e1}
    \frac
    {  \int_{  \Omega  }  |\nabla \varphi_0|^2  + \frac{1}{\rho_0^2 } |\nabla w_0|^2  }
    {  \int_{  \Omega  }  | \varphi_0|^2  + \frac{1}{\rho_0^2 } | w_0|^2  }
    \leq \frac{C}{\rho_0^2}  {\mathcal F}^4,
\end{equation}
with $C>0$ only depending on $M_0$, $M_1$, $s_0$, $\alpha_0$,
$\alpha_1$, $\gamma_0$ and $ \frac{\rho_0}{h}$. {}From
\eqref{LPS-p7-e4} and \eqref{LPS-p9-e1}, there exists $C>0$ only
depending on $M_0$, $M_1$, $s_0$, $\alpha_0$, $\alpha_1$,
$\gamma_0$ and $ \frac{\rho_0}{h}$, such that if we further assume
$\rho \leq \left ( \frac{1}{72C {\mathcal F}^4 } \right )^{
\frac{p}{p-2} }\rho_0$, where $p > 2$ is as in \eqref{LPS-p7-e4},
then
\begin{equation}
  \label{LPS-p9-e2}
    \frac
    {\int_{  \Omega \setminus \Omega_{  \frac{5}{\theta}\rho  }  }  |\varphi_0|^2 + \frac{1}{\rho_0^2 }|w_0|^2  }
    { \int_{\Omega}
    |\varphi_0|^2 +\frac{1}{\rho_0^2}|w_0|^2  }
    \leq
    \frac{1}{72} .
\end{equation}
Therefore, {}from \eqref{LPS-p4-e1}, \eqref{LPS-p6-e4},
\eqref{LPS-p7-e1} and \eqref{LPS-p9-e2}, we have
\begin{equation}
  \label{LPS-p9-e3}
    \left (
    \frac
    {\rho^2  \int_{B_{\rho}(x)}  E^2  }
    {\int_{\Omega}
    |\varphi_0^*|^2 +\frac{1}{\rho_0^2}|w_0^*|^2 }
    \right )^{\tau^L}
    \geq
    C \left ( \frac{\rho}{\rho_0}\right )^6
\end{equation}
and, by \eqref{LPS-p6-e3},
\begin{equation}
  \label{LPS-p10-e1}
    \int_{B_{\rho}(x)}  E^2
    \geq
    C \left ( \frac{\rho}{\rho_0}\right )^{ \frac{6}{\tau^L}  }
    \frac{1}{\rho^2}
    \int_{\Omega}
    |\varphi_0|^2 +\frac{1}{\rho_0^2}|w_0|^2 ,
\end{equation}
with $C>0$ only depending on $M_0$, $M_1$, $s_0$, $\alpha_0$,
$\alpha_1$, $\gamma_0$ and $ \frac{\rho_0}{h}$.

The integral on the right hand side of \eqref{LPS-p10-e1} can be
estimated {}from below first by using \eqref{LPS-p9-e1}, namely
\begin{equation}
  \label{LPS-p10-e2}
    \int_{B_{\rho}(x)}  E^2
    \geq
    \frac{C}{ {\mathcal F}^4  } \left ( \frac{\rho}{\rho_0}\right )^{ \frac{6}{\tau^L} -2 }
    \int_{\Omega}
    |\nabla \varphi_0|^2 +\frac{1}{\rho_0^2}|\nabla w_0|^2 ,
\end{equation}
and then by Poincar\'{e} inequality, obtaining
\begin{equation}
  \label{LPS-p10-e3}
    \int_{B_{\rho}(x)}  E^2
    \geq
    \frac{C}{ {\mathcal F}^4  } \left ( \frac{\rho}{\rho_0}\right )^{ \frac{6}{\tau^L} -2 }
    \int_{\Omega}
    E^2 ,
\end{equation}
with $C>0$ only depending on $M_0$, $M_1$, $s_0$, $\alpha_0$,
$\alpha_1$, $\gamma_0$ and $ \frac{\rho_0}{h}$. Hence
\eqref{eq:LPS} holds for $\rho\leq \overline{\gamma} \rho_0$, with
$\overline{\gamma}$ depending on $M_0$, $M_1$, $s_0$, $\alpha_0$,
$\alpha_1$, $\gamma_0$ and $ \frac{\rho_0}{h}$. If $\rho\geq
\overline{\gamma} \rho_0$, then the thesis follows \textit{a
fortiori}.
\end{proof}

\medskip

In order to prove Proposition \ref{prop:dopo_dualità}, let us
introduce the following Lemma.
\begin{lem}
   \label{lem:prima_dualità}
Under the hypotheses of Proposition \ref{prop:dopo_dualità}, let
us assume that $\varphi_{|\partial \Omega}\in H^1(\partial \Omega,
\R^2)$ and $w_{|\partial \Omega}\in H^1(\partial \Omega)$. Then
there exists a positive constant $C$ only depending on
$M_0$, $M_1$,
$\alpha_0$, $\alpha_1$, $\gamma_0$, $\frac{\rho_0}{h}$, such that
\begin{equation}
  \label{eq:prima_dualità}
    \|\overline{M}\|_{L^{2}(\partial \Omega,\R^2)}+\rho_0 \|\overline{Q}\|_{L^{2}(\partial \Omega)}\leq C\rho_0^2
        \left(\|\varphi_0\|_{H^{1}(\partial \Omega,\R^2)}+\frac{1}{\rho_0} \|w_0\|_{H^{1}(\partial \Omega)}
        \right).
\end{equation}
\end{lem}

\begin{proof}[Proof of Proposition \ref{prop:dopo_dualità}]
For brevity, we shall write $\varphi$, $w$ instead of $\varphi_0$, $w_0$ respectively. Let us consider the standard Dirichlet-to-Neumann map
\begin{equation*}
\Lambda = \Lambda_{\mathbb P,S}:
H^{1/2}(\partial\Omega,\R^2)\times
H^{1/2}(\partial\Omega)\rightarrow
H^{-1/2}(\partial\Omega,\R^2)\times H^{-1/2}(\partial\Omega),
\end{equation*}
\begin{equation*}
\Lambda(g_1,g_2) = ((\mathbb P\nabla \varphi) n, S(\varphi +
\nabla w)\cdot n),
\end{equation*}
where $(\varphi,w)\in H^{1}(\Omega,\R^2)\times H^{1}(\Omega)$ is
the unique solution to the Dirichlet problem
\begin{center}
\( {\displaystyle \left\{
\begin{array}{lr}
     \mathrm{\divrg}( S(\varphi+\nabla w))=0
      & \mathrm{in}\ \Omega,
        \vspace{0.25em}\\
      \mathrm{\divrg}( {\mathbb P}  \nabla \varphi)-S(\varphi+\nabla w)=0, & \mathrm{in}\ \Omega,
          \vspace{0.25em}\\
      \varphi= g_1,
      & \mathrm{on}\ \partial \Omega,
        \vspace{0.25em}\\
      w = g_2, &\mathrm{on}\ \partial
      \Omega.
          \vspace{0.25em}\\
\end{array}
\right. } \) \vskip -7.5em
\begin{eqnarray}
& & \label{eq:Dirichlet1}\\
& & \label{eq:Dirichlet2}\\
& & \label{eq:Dirichlet3}\\
& & \label{eq:Dirichlet4}
\end{eqnarray}
\end{center}
Here the norm in the domain of $\Lambda$ is normalized by
\begin{equation*}
\|(g_1,g_2)\|_{H^{1/2}(\partial\Omega,\R^2)\times
H^{1/2}(\partial\Omega)} = \|g_1\|_{H^{1/2}(\partial\Omega,\R^2)}
+ \rho_0^{-1}  \|g_2\|_{H^{1/2}(\partial\Omega)}
\end{equation*}
and similar normalizations will be implied in the sequel for other
norms in the domain of $\Lambda$ and in the codomain of its
adjoint $\Lambda^*$, whereas the norm in the codomain of $\Lambda$
is normalized by
\begin{equation*}
\|(h_1,h_2)\|_{H^{-1/2}(\partial\Omega,\R^2)\times
H^{-1/2}(\partial\Omega)} =
\|h_1\|_{H^{-1/2}(\partial\Omega,\R^2)} + \rho_0
\|h_2\|_{H^{-1/2}(\partial\Omega)}
\end{equation*}
and similar normalizations will be implied in the sequel for other
norms in the codomain of $\Lambda$ and in the domain of its
adjoint $\Lambda^*$.

Let us set
\begin{equation*}
E = H^{1}(\partial\Omega,\R^2)\times H^{1}(\partial\Omega),\qquad
F = L^2(\partial\Omega,\R^2)\times L^2(\partial\Omega).
\end{equation*}

By Lemma \ref{lem:prima_dualità} we know that the map $\Lambda$
can be defined as a bounded linear operator with domain $E$ and
codomain $F$, precisely
\begin{equation}
   \label{eq:D-Nseconda}
\Lambda: E \rightarrow F,
\end{equation}
\begin{equation}
  \label{eq:D-N-E-F}
\|\Lambda(g_1,g_2)\|_F\leq C \rho_0^2 \|(g_1,g_2)\|_E,
\end{equation}
where we recall that the norms in $E$ and $F$, according to the
above convention, are  defined as follows
\begin{equation*}
\|(g_1,g_2)\|_E = \|g_1\|_{H^1(\partial\Omega,\R^2)} + \rho_0^{-1}
\|g_2\|_{H^{1}(\partial\Omega)},
\end{equation*}
\begin{equation*}
\|(h_1,h_2)\|_F = \|h_1\|_{L^2(\partial\Omega,\R^2)} + \rho_0
\|h_2\|_{L^2(\partial\Omega)}.
\end{equation*}
Let us consider the adjoint $\Lambda^*$ of the
Dirichlet-to-Neumann map
\eqref{eq:D-Nseconda}--\eqref{eq:D-N-E-F}. Since $F$ is a
reflexive space, the domain of the adjoint operator $D(\Lambda^*)$
can be estended by density to all of $F'$,
\begin{equation*}
\Lambda^*: F' \rightarrow E'
\end{equation*}
\begin{equation}
  \label{eq:adjoint}
<\Lambda^*(h_1,h_2),(g_1,g_2)>_{E',E} =
<(h_1,h_2),\Lambda(g_1,g_2)>_{F',F} \quad \forall (g_1,g_2) \in E,
\forall (h_1,h_2)\in F'.
\end{equation}

By \eqref{eq:D-N-E-F}--\eqref{eq:adjoint}, we have
\begin{equation}
  \label{eq:cont_adjoint}
\|\Lambda^*(h_1,h_2)\|_{E'} \leq C\rho_0^2 \|(h_1,h_2)\|_{F'}
\quad \forall (h_1,h_2)\in F',
\end{equation}
Given any $(h_1,h_2)\in E\subset F\cong F'$, let us consider the
unique weak solution to the Dirichlet problem
\begin{center}
\( {\displaystyle \left\{
\begin{array}{lr}
     \mathrm{\divrg}( S(\psi+\nabla v))=0
      & \mathrm{in}\ \Omega,
        \vspace{0.25em}\\
      \mathrm{\divrg}( {\mathbb P}  \nabla \psi)-S(\psi+\nabla v)=0, & \mathrm{in}\ \Omega,
          \vspace{0.25em}\\
      \psi= h_1,
      & \mathrm{on}\ \partial \Omega,
        \vspace{0.25em}\\
      v = h_2, &\mathrm{on}\ \partial
      \Omega.
          \vspace{0.25em}\\
\end{array}
\right. } \) \vskip -7.5em
\begin{eqnarray}
& & \label{eq:newDirichlet1}\\
& & \label{eq:newDirichlet2}\\
& & \label{eq:newDirichlet3}\\
& & \label{eq:newDirichlet4}
\end{eqnarray}
\end{center}
By using the weak formulation of problems
\eqref{eq:Dirichlet1}--\eqref{eq:Dirichlet4} and
\eqref{eq:newDirichlet1}--\eqref{eq:newDirichlet4}, by the
symmetry properties of $S$ and $\mathbb P$, see
\eqref{eq:anto-11.1}-\eqref{eq:anto-11.2}, and by identifying  the
reflexive space $F$ with its dual space $F'$, we have
\begin{multline}
  \label{eq:autoagg_conti}
<\Lambda^*(h_1,h_2),(g_1,g_2)>_{E',E} = <(h_1,h_2),\Lambda(g_1,g_2)>_{F',F}=\\
=\int_{\partial\Omega}h_1\cdot(\mathbb P(\nabla\varphi)) n + h_2
S(\varphi+\nabla w)\cdot n =
\int_{\partial\Omega}\psi\cdot (\mathbb P(\nabla\varphi)) n + v S(\varphi+\nabla w)\cdot n = \\
=\int_\Omega\mathbb P\nabla \varphi\cdot\nabla\psi +
S(\varphi+\nabla w)\cdot(\psi+\nabla v) =
\int_\Omega\mathbb P\nabla \psi\cdot\nabla\varphi + S(\psi+\nabla v)\cdot(\varphi+\nabla w) =\\
= \int_{\partial\Omega}\varphi\cdot (\mathbb P(\nabla\psi)) n + w
S(\psi+\nabla v)\cdot n = \int_{\partial\Omega}g_1\cdot(\mathbb
P(\nabla\psi)) n + g_2 S(\psi+\nabla v)\cdot n,
\end{multline}
that is
\begin{equation}
  \label{eq:cont_adjoint2}
<\Lambda^*(h_1,h_2), (g_1,g_2)>_{E',E} =
<\Lambda(h_1,h_2),(g_1,g_2)>_{F',F}  \quad \forall (h_1,h_2),
(g_1,g_2)\in E.
\end{equation}
Therefore
\begin{equation}
   \label{eq:oper_coinc_aggiunto}
\Lambda^*(h_1,h_2) = \Lambda(h_1,h_2), \qquad \forall (h_1,h_2)\in
E\subset F\cong F'.
\end{equation}
By \eqref{eq:cont_adjoint}, we have
\begin{equation}
  \label{eq:cont_adjoint3}
\|\Lambda(h_1,h_2)\|_{H^{-1/2}(\partial\Omega,\R^2)\times
H^{-1/2}(\partial\Omega)} \leq C\rho_0^2
\|(h_1,h_2)\|_{L^{2}(\partial\Omega,\R^2)\times
L^{2}(\partial\Omega)}  \quad \forall (h_1,h_2)\in E.
\end{equation}
Since $E$ is dense in $L^{2}(\partial\Omega,\R^2)\times
L^{2}(\partial\Omega)$, the above inequality extends to
\begin{equation}
  \label{eq:cont_adjoint4}
\|\Lambda(h_1,h_2)\|_{H^{-1}(\partial\Omega,\R^2)\times
H^{-1}(\partial\Omega)} \leq C\rho_0^2
\|(h_1,h_2)\|_{L^2(\partial\Omega,\R^2)\times
L^2(\partial\Omega)},
\end{equation}
for every $(h_1,h_2)\in L^{2}(\partial\Omega,\R^2)\times
L^{2}(\partial\Omega)$.

\end{proof}

In order to derive Lemma \ref{lem:prima_dualità}, we need to
premise some notation and two auxiliary lemmas which were proved
in \cite{A-M-R02} and in \cite{MR03-JE} respectively.

Given the notation for the local representation of the boundary of
$\Omega$ introduced in Definition \ref{def:2.1}, let us set, for
$t<\rho_0$,
\begin{equation*}
R_t^+ = \Omega \cap R_{t,M_0t} = \{x=(x_1,x_2) \in \R^2 \ | \
|x_1|<t, \psi(x_1)<x_2 <M_0t \},
\end{equation*}
\begin{equation*}
\Delta_t = \{x=(x_1,x_2) \in \R^2 \ | \ |x_1|<t, x_2 =
\psi(x_1)\}.
\end{equation*}

The following Lemma is a straightforward consequence of Lemma 5.2
in \cite{A-M-R02} and of Lemma 4.3 in \cite{MR03-JE}, which were
established in general anisotropic setting.

\begin{lem}
   \label{lem:exlemma5.2_4.3}
Let $S \in
C^{0,1}(\overline{\Omega}, {\M}^{2})$ and
$\mathbb P \in C^{0,1}
(\overline{\Omega}, {\cal L} ({\M}^{2}, {\M}^{2}))$ given by
\eqref{eq:shearing-tensor},
\eqref{eq:bending-tensor} respectively,
with Lam\'e moduli satisfying
\eqref{eq:Lame-ell},
\eqref{eq:Lame-reg}.

For every $\widetilde{w} \in H^{3/2}(R_{\rho_0}^{+})$ such that ${\rm div}(S \nabla
 \widetilde{w}) \in L^{2}(R_{\rho_0}^{+})$ and $\widetilde{w}=| \nabla \widetilde{w} |=0$ on
  $\partial R_{\rho_0}^{+} \setminus \Delta_{\rho_0}$, we have
\begin{equation}
  \label{eq:ex6.8}
  \int_{\Delta_{\rho_0/2}} | S \nabla \widetilde{w} \cdot n|^{2}
  \leq
  C
  \left (h^2
  \int_{\Delta_{\rho_0}} | \nabla_{T} \widetilde{w}|^{2}
  +\frac{1}{\rho_0}
  \int_{R_{\rho_0}^{+}} h^2 | \nabla \widetilde{w}|^{2}
  +h\rho_0
  | \nabla \widetilde{w}||{\rm div}(S \nabla \widetilde{w})|
  \right
  ),
\end{equation}
where $C>0$ only depends on $M_0$, $\alpha_0$ and $\alpha_1$.

For every $\widetilde{\varphi} \in H^{3/2}(R_{\rho_0}^{+},\R^2)$ such that ${\rm div}(\mathbb P \nabla
 \widetilde{\varphi}) \in L^{2}(R_{\rho_0}^{+},\R^2)$ and $|\widetilde{\varphi}|=| \nabla \widetilde{\varphi} |=0$ on
  $\partial R_{\rho_0}^{+} \setminus \Delta_{\rho_0}$, we have
\begin{equation}
  \label{eq:ex6.9}
  \int_{\Delta_{\rho_0/2}} | (\mathbb P \nabla \widetilde{\varphi}) n|^{2}
  \leq
  C
  \left (h^6
  \int_{\Delta_{\rho_0}} | \nabla_{T} \widetilde{\varphi}|^{2}
  +\frac{1}{\rho_0}
  \int_{R_{\rho_0}^{+}} h^6 | \nabla \widetilde{\varphi}|^{2}
  +\rho_0h^3
  | \nabla \widetilde{\varphi}||{\rm div}(\mathbb P \nabla \widetilde{\varphi})|
  \right
  ),
\end{equation}
where $C>0$ only depends on $M_0$, $\alpha_0$, $\alpha_1$ and $\gamma_0$.
\end{lem}

\begin{proof}[Proof of Lemma \ref{lem:prima_dualità}]
We follow the lines of the proof of Proposition 5.1 in
\cite{A-M-R02}. As a first step, we assume that $\varphi \in
H^{3/2}(R_{\rho_0}^{+},\R^2)$ and $w \in H^{3/2}(R_{\rho_0}^{+})$.
Let us consider a cut-off function in $\R^{2}$
\begin{equation}
  \label{eq:cut1}
  \eta(x_1,x_{2})=\chi(x_1)\tau(x_{2}),
\end{equation}
where
\begin{equation}
  \label{eq:cut2}
  \chi \in C_{0}^{\infty} ( \R), \quad \chi(x_1)\equiv 1
  \ \textrm{if} \  |x_1|\leq\frac{\rho_0}{2}, \quad \chi(x_1)\equiv 0 \
  \textrm{if} \  |x_1|\geq\frac{3}{4}\rho_0,
\end{equation}
\begin{equation}
  \label{eq:cut3}
  \| \chi' \|_{\infty} \leq C_{1}
  \rho_0^{-1}, \quad
  \| \chi''\|_{\infty} \leq C_{1}
  \rho_0^{-2},
\end{equation}
\begin{equation}
  \label{eq:cut4}
  \tau \in C_{0}^{\infty} ( \R), \quad \tau( x_{2}) \equiv 1
  \ \textrm{if} \  |x _{2} |\leq\frac{M_0\rho_0}{2}, \quad \tau(x_{2} )\equiv 0 \
  \textrm{if} \  |x _{2}|\geq\frac{3}{4}M_0\rho_0,
\end{equation}
\begin{equation}
  \label{eq:cut5}
  \| \tau'\|_{\infty} \leq C_{2}
  \rho_0^{-1}, \quad
  \| \tau'' \|_{\infty} \leq C_{2}
  \rho_0^{-2},
\end{equation}
where $C_{1}$ is an absolute constant and $C_{2}$ is a constant
only depending on $M_0$.

Let
\begin{equation*}
    \widetilde{w} =\eta w,
\end{equation*}
\begin{equation*}
    \widetilde{\varphi} =\eta \varphi.
\end{equation*}
In view of equations \eqref{eq:intro-1}--\eqref{eq:intro-2}, it
will be useful in the sequel to rewrite ${\rm div}(S \nabla w)$ in
terms of first derivatives of $\varphi$ and ${\rm div}(\mathbb P
\nabla \varphi)$ in terms of first derivatives of $w$ and in terms
of $\varphi$
\begin{equation}
    \label{eq:abbasso_deriv_w}
 \mathrm{\divrg}(S\nabla w)=-  \mathrm{\divrg}(S\varphi),
\end{equation}
\begin{equation}
   \label{eq:abbasso_deriv_varphi}
      \mathrm{\divrg}({\mathbb P}\nabla \varphi) = S(\varphi+\nabla
      w).
\end{equation}
By \eqref{eq:abbasso_deriv_w}, it follows that ${\rm div}(S \nabla
 \widetilde{w}) \in L^{2}(R_{\rho_0}^{+})$ and by \eqref{eq:abbasso_deriv_varphi}, it follows that
${\rm div}(\mathbb P \nabla
 \widetilde{\varphi}) \in L^{2}(R_{\rho_0}^{+},\R^2)$. Therefore we can apply estimates \eqref{eq:ex6.8} and \eqref{eq:ex6.9} of Lemma
\ref{lem:exlemma5.2_4.3} to $\widetilde{w}$ and
$\widetilde{\varphi}$, respectively. Taking into account
\eqref{eq:cut1}--\eqref{eq:abbasso_deriv_varphi} we easily obtain
\begin{multline}
  \label{eq:stima_w}
  \int_{\Delta_{\rho_0/2}} | S \nabla w \cdot n|^{2}
  \leq \\
  Ch^2
  \left [
  \int_{\Delta_{\rho_0}} \left(| \nabla_{T} w|^{2} +\frac{w^2}{\rho_0^2}\right)
  +\frac{1}{\rho_0}
  \int_{R_{\rho_0}^{+}} \left(| \nabla w|^{2} + \frac{w^2}{\rho_0^2} +
    \rho_0^2 | \nabla \varphi|^{2} + |\varphi|^2
  \right)
  \right
  ],
\end{multline}
\begin{multline}
  \label{eq:stima_varphi}
  \int_{\Delta_{\rho_0/2}} | \mathbb P \nabla \varphi \cdot n|^{2}
  \leq \\
  Ch^6
  \left [
  \int_{\Delta_{\rho_0}} \left(| \nabla_{T} \varphi|^{2} +\frac{|\varphi|^2}{\rho_0^2}\right)
  +\frac{1}{\rho_0}
  \int_{R_{\rho_0}^{+}} \left(| \nabla \varphi|^{2} + \frac{|\varphi|^2}{\rho_0^2} +
    \frac{|\nabla w|^2}{\rho_0^2}
  \right)
  \right
  ],
\end{multline}
where $C>0$ only depends on $M_0$, $\alpha_0$, $\alpha_1$ and
$\gamma_0$.

By \eqref{eq:stima_w} and \eqref{eq:stima_varphi} we have
\begin{multline}
  \label{eq:stima_w+varphi}
  \int_{\Delta_{\rho_0/2}} | \mathbb P \nabla \varphi \cdot n|^{2}
    +\rho_0^2|S(\varphi+\nabla w)\cdot n|^2
  \leq \\
  Ch^6
  \left [
  \int_{\Delta_{\rho_0}} \left(| \nabla_{T} \varphi|^{2} +\frac{|\varphi|^2}{\rho_0^2}
    + \frac{|\nabla_{T} w|^{2}}{\rho_0^2}+ \frac{w^{2}}{\rho_0^4}\right)
  +\frac{1}{\rho_0}
  \int_{R_{\rho_0}^{+}} \left(| \nabla \varphi|^{2} + \frac{|\varphi|^2}{\rho_0^2} +
    \frac{w^{2}}{\rho_0^4} + \frac{|\nabla w|^2}{\rho_0^2}
  \right)
  \right
  ],
\end{multline}
where $C>0$ only depends on $M_0$, $\alpha_0$, $\alpha_1$ and
$\gamma_0$.

The hypotheses $\varphi \in H^{3/2}(R_{\rho_0}^{+},\R^2)$, $w \in
H^{3/2}(R_{\rho_0}^{+})$ can be removed by following the lines of
the approximation argument used in Step 3 of \cite[Lemma
5.2]{A-M-R02} and \cite[Lemma 4.3]{MR03-JE} respectively,
obtaining again \eqref{eq:stima_w+varphi}. Finally, by
\eqref{eq:stima_w+varphi} and the well-posedness of the Dirichlet
problem \eqref{eq:Dirichlet1}--\eqref{eq:Dirichlet4}, inequality
\eqref{eq:prima_dualità} follows.

\end{proof}

\bigskip
\bibliographystyle{plain}

\end{document}